\documentclass[preprint,aos]{imsart}

\RequirePackage{amsthm,amsmath,amsfonts,amssymb}
\RequirePackage[numbers]{natbib}
\RequirePackage[colorlinks,citecolor=blue,urlcolor=blue]{hyperref}
\RequirePackage{graphicx}

\startlocaldefs
\theoremstyle{plain}

\newtheorem{theorem}{Theorem}[section]
\newtheorem{lemma}[theorem]{Lemma}
\theoremstyle{remark}

\newtheorem*{example}{Example}

\endlocaldefs

\begin{document}

\begin{frontmatter}
\title{A Proof of consistency and model-selection optimality\\ on the empirical Bayes method}
\runtitle{Consistency and optimality of the parametric empirical Bayes}

\begin{aug}
\author[A]{\fnms{Dye SK}~\snm{Sato}\ead[label=e1]{sato.daisuke.44p@st.kyoto-u.ac.jp}\orcid{0000-0001-8809-9268}}
\and
\author[A]{\fnms{Yukitoshi}~\snm{Fukahata}\ead[label=e2]{fukahata.yukitoshi.3e@kyoto-u.ac.jp}}
\address[A]{Disaster Prevention Research Institute, Kyoto University, Gokasho, Uji, Kyoto 611-0011, Japan\printead[presep={,\ }]{e1,e2}}
\end{aug}

\begin{abstract}

We study the consistency and optimality of the maximum marginal likelihood estimate (MMLE) in the hyperparameter inference for large-degree-of-freedom models. 
We perform main analyses within the exponential family, where the natural parameters are hyperparameters. 
First, we prove the consistency of the MMLE for the general linear models when estimating the scales of variance in the likelihood and prior. 
The proof is independent of the number ratio of data to model parameters and excepts the ill-posedness of the associated regularized least-square model-parameter estimate that is shown asymptotically unbiased. 
Second, we generalize the proof to other models with a finite number of hyperparameters. 
We find that the extensive properties of cost functions in the exponential family generally yield the consistency of the MMLE for the likelihood hyperparameters. 
Besides, we show the MMLE asymptotically almost surely minimizes the Kullback-Leibler divergence between the prior and true predictive distributions even if the true data distribution is outside the model space, under the hypothetical asymptotic normality of the predictive distributions applicable to non-exponential model families. 
Our proof validates the empirical Bayes method using the hyperparameter MMLE in the asymptotics of many model parameters, ensuring the same qualification for the empirical-cross-entropy cross-validation. 
\end{abstract}

\begin{keyword}[class=MSC]
\kwd[Primary ]{60B10}
\kwd{62H10}
\kwd[; secondary ]{62G05}
\end{keyword}

\begin{keyword}
\kwd{Empirical Bayes method}
\kwd{consistency}
\end{keyword}

\end{frontmatter}

\section{Introduction}

Explaining data with a statistical model could necessitate the validation of the statistical model itself~\citep{akaike1980likelihood}. 
Besides the parameters of the statistical model (model parameters), Bayesian hierarchical modelling~\citep{gelman2013bayesian} introduces the parameters that specify the structure of the statistical model (hyperparameters), enabling the estimation of an appropriate model structure from the data. 
Large-degree-of-freedom models containing a large number of model parameters also learn the model structure from the data in a model-parameter space of high dimension comparable to or larger than the number of data, where the hyperparameters often appear in the prior to solve the instability problem in the likelihood function with large degrees of freedom. 
The hierarchical parameter space raises a new need for meta-analysis~\citep{cooper2019handbook} to select a reasonable statistical model from the model family the hyperparameters represent. 

The validity of the statistical model is often measured by the marginal likelihood of the hyperparameters, integrating the product of the likelihood and prior in the model-parameter space~\citep{good1965estimation}. 
The maximum marginal likelihood estimate (MMLE) thus has a wide range of application, with various names of the marginal likelihood of the hyperparameters, such as the type II likelihood~\citep{good1965estimation}, likelihood of a model~\citep{akaike1980likelihood}, and model evidence~\citep{bishop2006pattern}, which is also the empirical cross entropy~\citep{friedman2001elements} for one data set. 
The marginal likelihood of the hyperparameters also describes the probability of expected data for given hyperparameters, termed the prior predictive distribution~\citep{gelman2013bayesian}. 
The MMLE with a number correction of the hyperparameters is then proposed, called Akaike's Bayesian information criterion~\citep[ABIC;][]{akaike1980likelihood},  
aiming the minimum Kullback-Leibler divergence of the prior predictive from the true data distribution. 

On the other, the use of the MMLE is also called the empirical Bayes method~\citep{maritz2018empirical,robbins1992empirical} often criticized for its point estimation nature~\citep{carlin2008bayesian}. 
The censure may also be directed to the rationale of the MMLE. 
In the dimension decision of non-hierarchical models, Akaike's information criterion~\citep[AIC;][]{akaike1974new} is proposed for the minimum Kullback-Leibler divergence between the prior and true predictive distributions~\citep[model-selection efficiency;][]{shibata1981optimal} similarly to ABIC and fails to earn the convergence in probability to the true value (consistency), leading to Bayesian information criterion~\citep[BIC;][]{schwarz1978estimating}. BIC is an approximated MMLE for the model dimension decision, shown consistent yet less optimal than AIC for the Kullback-Leibler minimization in data prediction~\citep{shibata1981optimal,yang2005can}. 
For infinite-degree-of-freedom models with a fixed number of hyperparameters, ABIC and the MMLE are respectively the counterparts of AIC and BIC but are identical, obscuring theoretical explanations on how the MMLE performs in the data predictability and knowledge acquisition on the true hyperparameters. 
In this study, primarily within the exponential family, 
we prove the consistency of the MMLE for the hyperparameter inference in models with model parameters comparable to or greater than data, where we further see the consistency can be compatible with the model-selection efficiency. 

\section{Problem setting and preliminary}
In this section, we define the inversion model we consider and introduce several relations useful for the proof of consistency. 
\subsection{Model family and selection}
We conduct the analysis in the exponential family~\citep{barndorff2014information} except in \S\ref{sec:gen_efficiency_app}. 
Suppose inverting given data ${\bf d}\in\mathbb{R}^N$ to the following probabilistic model: 
\begin{equation}
\begin{aligned}
    P({\bf d}|{\bf a};\boldsymbol\beta_{\rm d})&\simeq \exp[-\boldsymbol\beta_{\rm d}^{\rm T} {\bf U}({\bf d},{\bf a})+F_{\rm d}(\boldsymbol\beta_{\rm d})]
    \\
    P({\bf a};\boldsymbol\beta_{\rm a})&= \exp[-\boldsymbol\beta_{\rm a} ^{\rm T}{\bf V}({\bf a})+F_{\rm a}(\boldsymbol\beta_{\rm a})]
\end{aligned}
\label{eq:modeldefinition}
\end{equation}
where ${\bf a}\in\mathbb{R}^M$ denotes the model parameters, 
$\boldsymbol\beta_{\rm d}\in (0,\infty)^{n_{\rm d}}$ and $\boldsymbol\beta_{\rm a}\in (0,\infty)^{n_{\rm a}}$ the hyperparameters of $
P({\bf d}|{\bf a};\boldsymbol\beta_{\rm d})$ and $P({\bf a};\boldsymbol\beta_{\rm a})$, which are respectively weighted by the cost functions 
${\bf U}({\bf d},{\bf a})\in [0,\infty)^{n_{\rm d}}$ and ${\bf V}({\bf a})\in [0,\infty)^{n_{\rm a}}$ and normalized by $F_{\rm a}(\boldsymbol\beta_{\rm a})$ and $F_{\rm d}(\boldsymbol\beta_{\rm d})$, 
and ${}^{\rm T}$ represents the transpose. 
We presume $U_i$ and $V_i$ may take more than two values to eliminate the possibility of model degeneracy. 
Equations~(\ref{eq:modeldefinition}) with unfixed $\boldsymbol\beta$ [$:=(\boldsymbol\beta_{\rm d},\boldsymbol\beta_{\rm a})^{\rm T}$] provide a variety of the probabilistic models, which constitute a model family termed the exponential family~\citep{gelman2013bayesian}. 
This modeling redundancy enables us to select an appropriate model of ${\bf d}$ and ${\bf a}$ as meta-inference of the hyperparameters $\boldsymbol\beta$ (natural parameters). 
When considering the prior as a stochastic model-parameter generator ${\bf a}\sim P({\bf a};\boldsymbol\beta_{\rm a})$, the above probabilistic model describes the data generation by 
$P({\bf d};\boldsymbol\beta)$~\citep[the prior predictive distribution;][]{gelman2013bayesian}: 
\begin{equation}
    P({\bf d};\boldsymbol\beta)=\int d{\bf a}P({\bf d}|{\bf a};\boldsymbol\beta_{\rm d})P({\bf a};\boldsymbol\beta_{\rm a}).
\end{equation}

Hereafter, our interest is in approximating $P({\bf d};\boldsymbol\beta)$ to the true data distribution $Q({\bf d})$. 
$Q({\bf d})$ necessarily exists as the original data generator [${\bf d}\sim Q({\bf d})$], although not always within the model family of $P({\bf d};\boldsymbol\beta)$. 
When there is a true hyperparameter vector $\boldsymbol\beta_{\rm 0}=(\boldsymbol\beta_{{\rm d}0},\boldsymbol\beta_{{\rm a}0})^{\rm T}$ such that $P({\bf d};\boldsymbol\beta_{\rm 0})=Q({\bf d})$, a hyperparameter estimate $\hat {\boldsymbol\beta}$ is phrased to be consistent if satisfies the following relation for an arbitrarily small positive constant $\epsilon$: 
\begin{equation}
    \lim_{N\to\infty} {\rm Prob}[|\hat{\boldsymbol\beta}-\boldsymbol\beta_{\rm 0}|>\epsilon]=0 \hspace{20pt}({\rm consistency})
\end{equation}
where the realization probability of a proposition ${\rm Prob}[\cdot]$ is evaluated by $Q({\bf d})$, and $|\cdot|$ represents the L2 norm for vectors. 
In other words, the consistent estimate converges in probability to the true value in the limit of a large number of data $N$. 
We treat $M/N=const.$ (here called large degrees of freedom, or for brevity, ``nonparametric cases'') with a large number of model parameters $M$, which is complementary to the assumption of constant $M$ (``parametric cases'') in the proof of consistency on the maximum likelihood methods resulting in $M/N\to0$. 
The number of hyperparameters are meanwhile supposed to be finite. 
We specifically consider the maximum marginal likelihood estimate (MMLE) of $\boldsymbol\beta$: 
\begin{equation}
    \hat{\boldsymbol\beta}= {\rm argmax}_{\boldsymbol\beta} P({\bf d};\boldsymbol\beta).
\end{equation}
In summary, the thesis of this study is the consistency of the maximum marginal likelihood estimator for the natural parameters of exponential families with large $M$ ($\propto N$), which is generalized with the model-selection efficiency and other models in \S\ref{sec:gen_efficiency} and \S\ref{sec:gen_efficiency_app}. 
The natural parameters are identified with hyperparameters in this paper, as far as we consider the exponential family. 
In the proof of consistency, we presume the premised existence of $\boldsymbol\beta_0$.

\subsection{Properties of the MMLE in the natural-parameter inference}
\label{sec:MMLEinGibbsians}
The probability distributions in eqs.~(\ref{eq:modeldefinition}) within the exponential family are also termed Gibbs ensembles, which have several features useful for our proof of consistency~\citep{landau1994statistical}. 
One is the cumulant-generating nature of the normalization functions (called free entropies); the $n$-th derivatives of the minus free entropies $F_{\rm a}(\boldsymbol\beta_{\rm a})$ and $F_{\rm d}(\boldsymbol\beta_{\rm d})$, 
or equivalently, the free entropy $F(\boldsymbol\beta):=F_{\rm d}(\boldsymbol\beta)+F_{\rm a}(\boldsymbol\beta)$ of $P({\bf a},{\bf d};\boldsymbol\beta)$, 
are the $n$-th cumulants of the associated cost functions:
\begin{equation}
\begin{aligned}
    \frac{\partial F}{\partial \beta_{i}}=\langle \mathcal E_i\rangle_{\boldsymbol\beta}
    \\
    \frac{\partial^2 F}{\partial \beta_{i}\partial\beta_{j}}=-{\rm Cov}[\mathcal E_i,\mathcal E_j;{\boldsymbol\beta}]
\end{aligned}
\end{equation}
where $\langle \cdot \rangle_{\boldsymbol\beta}:=\int d{\bf a}d{\bf d}\cdot P({\bf d},{\bf a};\boldsymbol\beta)$ represents averaging by $P({\bf d},{\bf a};\boldsymbol\beta)$, 
and
${\rm Cov}[f,g;\boldsymbol\beta]:=\langle fg\rangle_{\boldsymbol\beta}-\langle f \rangle_{\boldsymbol\beta}\langle g\rangle_{\boldsymbol\beta}$ denotes the associated covariance for any random variables $f$ and $g$; $\boldsymbol{\mathcal E}=({\bf U},{\bf V})^{\rm T}$ is introduced for brevity. 
Higher order cumulants are similarly generated, and 
the cross cumulants of $U_i$ and $V_j$ are all zero. 

Another remarkable feature comes from the conjugacy of the Gibbsian models; 
the Gibbsian prior is the conjugate prior of the Gibbsian likelihood, and they provide the following Gibbsian posterior $P({\bf a}|{\bf d};\boldsymbol\beta)$:
\begin{equation}
    P({\bf a}|{\bf d};\boldsymbol\beta)=\frac{P({\bf a},{\bf d};\boldsymbol\beta)}{P({\bf d};\boldsymbol\beta)}
    =\exp[-\boldsymbol\beta^{\rm T} \boldsymbol{\mathcal E}({\bf a},{\bf d})+F_{\rm pos}(\boldsymbol\beta,{\bf d})],
\end{equation}
where
the free entropy $F_{\rm pos}(\boldsymbol\beta,{\bf d})$ of the posterior is given by the following relation:
\begin{equation}
F_{\rm pos}(\boldsymbol\beta,{\bf d}):=
F(\boldsymbol\beta)
-\ln P({\bf d};\boldsymbol\beta).
\label{eq:defofposfreeentropy}
\end{equation}
$F_{\rm pos}$ is the posterior cumulant generator, similarly to $F_{\rm d}$ and $F_{\rm a}$ (and $F$): 
\begin{equation}
\begin{aligned}
    \frac{\partial F_{\rm pos}}{\partial \beta_{i}}=\langle {\mathcal E}_i\rangle_{\boldsymbol\beta|{\bf d}}
    \\
    \frac{\partial^2 F_{\rm pos}}{\partial \beta_{i}\partial\beta_{j}}=-{\rm Cov}[\mathcal E_i,\mathcal E_j;{\boldsymbol\beta}|{\bf d}]
\end{aligned}
\end{equation}
where $\langle \cdot \rangle_{\boldsymbol\beta|{\bf d}}:=\int d{\bf a}\cdot P({\bf a}|{\bf d};\boldsymbol\beta)$ represents posterior averaging by $P({\bf a}|{\bf d};\boldsymbol\beta)$, 
and ${\rm Cov}[f,g;\boldsymbol\beta|{\bf d}]:=\langle fg\rangle_{\boldsymbol\beta|{\bf d}}-\langle f \rangle_{\boldsymbol\beta|{\bf d}}\langle g\rangle_{\boldsymbol\beta|{\bf d}}$ denotes the posterior covariance. 
Equation~(\ref{eq:defofposfreeentropy}) indicates that the 
marginal posterior of the hyperparameters is the free entropy difference between the data and model-parameter inversion model $P({\bf a},{\bf d};\boldsymbol\beta)$ and the posterior $P({\bf a}|{\bf d};\boldsymbol\beta)$~\citep{iba1996learning}; 
\begin{equation}
    \ln P({\bf d};\boldsymbol\beta)=F(\boldsymbol\beta)-F_{\rm pos}(\boldsymbol\beta,{\bf d}).
\label{eq:logPPMasfreenetropydiff}
\end{equation}
The increase of the free entropy corresponds to the decrease in the (Shannon) entropy in the concept of the second law of thermodynamics~\citep{landau1994statistical}. 
Equation~(\ref{eq:logPPMasfreenetropydiff}) then signifies that the 
log marginal likelihood of the hyperparameteres represents the reduction of uncertainty (i.e. entropy) by an observation, although not a simple comparison given their different sample-spaces (${\bf a}$, ${\bf d}$) and ${\bf d}$. 
For discrete ${\bf d}$, the left-hand side of eq.~(\ref{eq:logPPMasfreenetropydiff}) becomes generally negative (with a probability instead of the probability density) always increasing the free entropy. 

Given the cumulant generating nature of the free entropy, eq.~(\ref{eq:logPPMasfreenetropydiff}) means the marginal likelihoof of the hyperparameters is the cumulant-difference generator. 
Its first derivative is 
\begin{equation}
    \frac{\partial}{\partial\beta_{i}}\ln P ({\bf d};\boldsymbol\beta)= \langle {\mathcal E}_i\rangle_{\boldsymbol\beta}-\langle {\mathcal E}_i\rangle_{\boldsymbol\beta|{\bf d}}.
\label{eq:firstderivlnP}
\end{equation}
Since the mode in an open interval, such as the present case $0<\beta_i<\infty$, satisfies the stationary condition, 
eq.~(\ref{eq:firstderivlnP}) reveals that the MMLE requires the observation invariance of the cost function mean~\citep{iba1996learning}: 
\begin{equation}
    \langle {\mathcal E}\rangle_{\hat{\boldsymbol\beta}}=\langle {\mathcal E}\rangle_{\hat{\boldsymbol\beta}|{\bf d}}.
\label{eq:Ibarel}
\end{equation}
For the second derivative, we have
\begin{equation}
\begin{aligned}
    -\frac{\partial^2}{\partial \beta_i\partial \beta_j} \ln P({\bf d};\boldsymbol\beta)
    &=
    {\rm Cov}[\mathcal E_i,\mathcal E_j;\boldsymbol\beta]-{\rm Cov}[\mathcal E_i,\mathcal E_j;\boldsymbol\beta|{\bf d}].
\end{aligned}
\label{eq:secondderivlnP}
\end{equation}
Equation~(\ref{eq:secondderivlnP}) shows that the concavity of the marginal likelihood of the hyperparameters is equivalent to
the posterior variance reduction of the cost functions. The higher-order cumulant differences are similarly generated. 

Averaging the extremal condition eq.~(\ref{eq:Ibarel}) of the MMLE provides an useful relation for our proof of consistency: 
\begin{equation}
    \langle \langle \boldsymbol{\mathcal E}\rangle_{\boldsymbol\beta|{\bf d}}\rangle_{\boldsymbol\beta_0}=\langle \boldsymbol{\mathcal E}\rangle_{\boldsymbol\beta}.
     \label{eq:avIbarel}
\end{equation}
Remarkably, eq.~(\ref{eq:avIbarel}) comprehends the mean invariance of $\boldsymbol{\mathcal E}$ under the probability decomposition $P({\bf a},{\bf d};\boldsymbol\beta)=P({\bf a}|{\bf d};\boldsymbol\beta)P({\bf d};\boldsymbol\beta)$ for $\boldsymbol\beta=\boldsymbol\beta_0$: 
\begin{equation}
    \langle \langle \boldsymbol{\mathcal E}\rangle_{\boldsymbol\beta_0|{\bf d}}\rangle_{\boldsymbol\beta_0}=\langle \boldsymbol{\mathcal E}\rangle_{\boldsymbol\beta_0},
    \label{eq:avIbarel2}
\end{equation}
which is the rewritten stationary condition of the minimum cross entropy: 
\begin{equation}
    \left.\left\langle \frac{\partial}{\partial \boldsymbol\beta}\ln P({\bf d};\boldsymbol\beta)\right\rangle_{\boldsymbol\beta_0}\right|_{\boldsymbol\beta=\boldsymbol\beta_0}={\bf 0}.
    \label{eq:avIbarel2_probform}
\end{equation}
Equations~(\ref{eq:avIbarel}) and (\ref{eq:avIbarel2}) derive from the definitional identity of the cross entropy $\langle -\ln P({\bf d};\boldsymbol\beta)\rangle_{\boldsymbol\beta_0}$, 
which is the $\boldsymbol\beta$-dependent part of the Kullback-Leibler divergence of $P({\bf d};\boldsymbol\beta)$ from the true data generator $P({\bf d};\boldsymbol\beta_0)$:
\begin{equation}
    \mathcal D[P({\bf d}|\boldsymbol\beta)||P({\bf d}|\boldsymbol\beta_0)]:=\langle -\ln P({\bf d}|\boldsymbol\beta)\rangle_{\boldsymbol\beta_0}-\langle -\ln P({\bf d}|\boldsymbol\beta_0)\rangle_{\boldsymbol\beta_0},
\end{equation}
which is generally nonnegative and takes zero if and only if $P({\bf d}|\boldsymbol\beta)=P({\bf d}|\boldsymbol\beta_0)$ that corresponds to $\boldsymbol\beta=\boldsymbol\beta_0$. 
Equations~(\ref{eq:avIbarel}) and (\ref{eq:avIbarel2}) motivate us to prove the consistency of the MMLE for the natural parameters with the large-deviation properties of the cost functions. 

Likewise, averaging eq.~(\ref{eq:secondderivlnP}) by $P({\bf d};\boldsymbol\beta)$, we have
\begin{equation}
    -\left.\left\langle \frac{\partial^2}{\partial \beta_i\partial\beta_j} \ln P({\bf d};\boldsymbol\beta)\right\rangle_{\boldsymbol\beta_0}\right|_{\boldsymbol\beta=\boldsymbol\beta_0}= {\rm Cov}[\langle \mathcal E_i\rangle_{\boldsymbol\beta|{\bf d}},\langle \mathcal E_j\rangle_{\boldsymbol\beta|{\bf d}};\boldsymbol\beta]\succ 0,
    \label{eq:meanvariancereduction}
\end{equation}
where $\succ$ denotes the positive definiteness of the left-hand side minus the right-hand side. 
The calculation utilizes the decomposition of the covariance ${\rm Cov}[\cdot,\cdot;\boldsymbol\beta]$ into the interclass covariance ${\rm Cov}[\langle\cdot\rangle_{\boldsymbol\beta|{\bf d}},\langle\cdot\rangle_{\boldsymbol\beta|{\bf d}};\boldsymbol\beta]$ and intraclass covariance $\langle{\rm Cov}[\cdot,\cdot|{\bf d};\boldsymbol\beta]\rangle_{\boldsymbol\beta}$; for random variables $x$ and $y$,
\begin{equation}
{\rm Cov}[x,y;\boldsymbol\beta]=
{\rm Cov}[\langle x\rangle_{\boldsymbol\beta|{\bf d}},\langle y\rangle_{\boldsymbol\beta|{\bf d}};\boldsymbol\beta]
+
\langle{\rm Cov}[x,y|{\bf d};\boldsymbol\beta]\rangle_{\boldsymbol\beta}.
\label{eq:variancedecomposition}
\end{equation}
Equation~(\ref{eq:meanvariancereduction}) represents the posterior variance reduction of the cost functions $\boldsymbol{\mathcal E}$ on average. 
While the positive semidefiniteness of the Fisher information matrix is generally true without limited to the Gibbs distribution, its positive definiteness further follows that of the interclass variance in Gibbsian $P({\bf a},{\bf d};\boldsymbol\beta)$, as far as there is an $\langle E_i\rangle_{\boldsymbol\beta|{\bf d}}$ value such that $\min \langle E_i\rangle_{\boldsymbol\beta|{\bf d}}<\langle E_i\rangle_{\boldsymbol\beta|{\bf d}}<\infty$.

\section{Consistency of the hyperparameter MMLE in general linear models}
\label{eq:proof4gaussmodel}
Examples of the exponential family include the general linear models. 
Its two-hyperparameter form is often employed, associated with the regularized least squares~\citep{akaike1980likelihood,yabuki1992geodetic}: 
\begin{equation}
    \begin{aligned}
    {\bf d}|{\bf a};\sigma^2&\sim \mathcal N({\bf Ha},\sigma^2{\bf E}^{-1})
    \\
    {\bf a};\rho^2&\sim \mathcal N(0,\rho^2{\bf G}^\dag)\times u
    \end{aligned}
    \label{eq:twohyperparameterform}
\end{equation}
where $\sigma^2\in (0,\infty)$ and $\rho^2\in (0,\infty)$ are the hyperparameters of the likelihood and prior, respectively,
${\bf H}\in \mathbb R^{N\times M}$, ${\bf E}\in \mathbb R^{N\times N}$, and ${\bf G}\in \mathbb R^{M\times M}$ denote the coefficient matrices, and ${}^\dag$ represents the generalized inverse; 
a uniform filter $u$ is assumed over a sufficiently wide range within the null-eigenspace of ${\bf G}$ for the properness of the model-parameter prior. 
The hyperparameters $\sigma^2$ and $\rho^2$ here represent the scales of variance in the likelihood and prior of ${\bf a}$, respectively. 
This model is rewritten in a canonical form (eqs.~\ref{eq:modeldefinition}) with
$\beta_{\rm d}=1/\sigma^2$ and $\beta_{\rm a}=1/\rho^2$ (scales of precision) and the quadratic cost functions: 
\begin{equation}
    \begin{aligned}
        U&=\frac 1 2 ({\bf d}-{\bf Ha})^{\rm T}{\bf E}^{-1}({\bf d}-{\bf Ha})
        \\
        V&=\frac 1 2 {\bf a}^{\rm T}{\bf Ga}
    \end{aligned}
\end{equation}
The rank [${\rm rk}(\cdot)$] of ${\bf G}$ is hereafter denoted by $P:={\rm rk}({\bf G})>0$ [with ${\rm rk}(u)=M-P$], and ${\bf E}$ is presumed to be full-ranked. We also suppose $N+P>M$ for the well-posedness of the after-mentioned regularized least square estimates ${\bf a}_*$; otherwise ($N+P\leq M$), the regularized least square estimate could overfit the data, similarly to the least square for $N\leq M$. For the same reason, we assume ${\bf G}\not\propto{\bf H}^{\rm T}{\bf E}^{-1}{\bf H}$. 
Additionally, we limit our asymptotic analysis to the case of $P/M\to const.$ to focus on the two intrinsic large parameters of large-degree-of-freedom models: $N$ and $M$; constant $P/M$ and $P$ close to $M$ may be practical examples. 

Here we present the proof of consistency on the maximum marginal likelihood estimation of $\beta_{\rm d}$ and $\beta_{\rm a}$. 
The prior predictive distribution for this case is given as follows~\citep[e.g.][]{yabuki1992geodetic}:
\begin{equation}
\begin{aligned}
    \ln P({\bf d};\boldsymbol\beta)=&-\beta_{\rm d}[U_*({\bf d},\lambda)+\lambda V_*({\bf d},\lambda)]+\frac{N+P-M}2\ln \beta_{\rm d}
    \\
    &+\frac P 2 \ln\lambda -\frac 1 2 \ln |{\bf H}^{\rm T}{\bf E}^{-1}{\bf H}+\lambda{\bf G}|+c,
    \end{aligned}
    \label{eq:priorpredictivemodel}
\end{equation}
where $U_*({\bf d},\lambda):=U({\bf d},{\bf a}_*(\lambda,{\bf d}))$ and $V_*({\bf d},\lambda):=({\bf d},{\bf a}_*(\lambda,{\bf d}))$ are the cost-function values for the regularized least square estimate of the model parameters 
${\bf a}_*({\bf d},\lambda):=({\bf H}^{\rm T}{\bf E}^{-1}{\bf H}+\lambda {\bf G})^{-1}{\bf H}^{\rm T}{\bf E}^{-1}{\bf d}$ for given data ${\bf d}$ and $\lambda:=\beta_{\rm a}/\beta_{\rm d}(=\sigma^2/\rho^2)$, 
and $|\cdot|$ represents the determinant for matrices; $c$ represents a normalization constant independent of the data and hyperparameter values. 

\subsection{Local analysis}

The proof is designed to show the true hyperparameters $\boldsymbol\beta_0$ constitute the mode (besides the unique stationary point) of the marginal likelihood of the hyperparameters. 
In this subsection, we prove the existence of the stationary point within the $\epsilon$-neighborhood of the true hyperparameters; it is the proof of consistency provided the stationary point is unique in $\ln P({\bf d};\boldsymbol\beta)$ for given data ${\bf d}$, which leads to the general proof in the next subsection. 

We first find the stationary condition of the marginal likelihood of the hyperparameters $\partial_{\boldsymbol\beta} \ln P({\bf d};\boldsymbol\beta)={\bf 0}$ is almost surely satisfied in leading orders at $\boldsymbol\beta=\boldsymbol\beta_0$. 
The left hand side of eq.~(\ref{eq:Ibarel}) is now
\begin{equation}
\begin{aligned}
    \langle U\rangle_{\boldsymbol\beta}=\frac{N}{2\beta_{\rm d}}
    \\
    \langle V\rangle_{\boldsymbol\beta}=\frac{P}{2\beta_{\rm a}}
\end{aligned}
\label{eq:equipartition_UV}
\end{equation}
The right hand side of eq.~(\ref{eq:Ibarel}) for $\boldsymbol\beta=\boldsymbol\beta_0$ fluctuates around the mean identical to its left hand side (eq.~\ref{eq:avIbarel}) with the variance bounded as
\begin{equation}
    {\rm Var}[\langle U\rangle_{\boldsymbol\beta|{\bf d}};\boldsymbol\beta]
    ={\rm Var}[ U;\boldsymbol\beta]-{\rm Var}[ U|{\bf d};\boldsymbol\beta]
    \leq {\rm Var}[ U;\boldsymbol\beta]
    =\frac{N}{2\beta_{\rm d}^2},
    \label{eq:variancebounding_U}
\end{equation}
and similarly, 
\begin{equation}
    {\rm Var}[\langle V\rangle_{\boldsymbol\beta|{\bf d}};\boldsymbol\beta]\leq {\rm Var}[ V;\boldsymbol\beta]=\frac{P}{2\beta_{\rm a}^2},
    \label{eq:variancebounding_V}
\end{equation}
where ${\rm Var}[\cdot;\boldsymbol\beta]:=\langle \cdot:\cdot\rangle_{\boldsymbol\beta}$ and ${\rm Var}[\cdot|{\bf d};\boldsymbol\beta]:=\langle \cdot:\cdot\rangle_{\boldsymbol\beta|{\bf d}}$.
We use the decomposition of the variance ${\rm Var}[\cdot;\boldsymbol\beta]$ into the interclass variance ${\rm Var}[\langle\cdot\rangle_{\boldsymbol\beta|{\bf d}};\boldsymbol\beta]$ and intraclass variance $\langle{\rm Var}[\cdot|{\bf d};\boldsymbol\beta]\rangle_{\boldsymbol\beta}$ (a scalar form of eq.~\ref{eq:variancedecomposition}): 
\begin{equation}
    {\rm Var}[\cdot;\boldsymbol\beta]
    ={\rm Var}[\langle\cdot\rangle_{\boldsymbol\beta|{\bf d}};\boldsymbol\beta]+
    \langle{\rm Var}[\cdot|{\bf d};\boldsymbol\beta]\rangle_{\boldsymbol\beta}.
\end{equation}
Equations~(\ref{eq:avIbarel}) and (\ref{eq:equipartition_UV})--(\ref{eq:variancebounding_V}) with Chebyshev's inequality yield
\begin{equation}
\begin{aligned}
    {\rm Prob}\left[\left|\frac{\langle U\rangle_{\boldsymbol\beta_0|{\bf d}}}{\langle U\rangle_{\boldsymbol\beta_0}}-1\right|>\epsilon\right]
    \leq 
    \frac{{\rm Var}[\langle U\rangle_{\boldsymbol\beta_0|{\bf d}};\boldsymbol\beta_0]}{\epsilon^2\langle U\rangle_{\boldsymbol\beta_0}^2}
    \leq 
    \frac{{\rm Var}[U;\boldsymbol\beta_0]}{\epsilon^2\langle U\rangle_{\boldsymbol\beta_0}^2}
    = \frac{2}{N\epsilon^2}
    \\
    {\rm Prob}\left[\left|\frac{\langle V\rangle_{\boldsymbol\beta_0|{\bf d}}}{\langle V\rangle_{\boldsymbol\beta_0}}-1\right|>\epsilon\right]
    \leq 
    \frac{{\rm Var}[\langle V\rangle_{\boldsymbol\beta_0|{\bf d}};\boldsymbol\beta_0]}{\epsilon^2\langle V\rangle_{\boldsymbol\beta_0}^2}
    \leq 
    \frac{{\rm Var}[V;\boldsymbol\beta_0]}{\epsilon^2\langle V\rangle_{\boldsymbol\beta_0}^2}
    = \frac{2}{P\epsilon^2}
    \end{aligned}
    \label{eq:fluctuationprobabilitybound}
\end{equation}
and therefore the desired relation:
\begin{equation}
    \lim_{N\to\infty}
    {\rm Prob}\left[\left|\frac{\langle U\rangle_{\boldsymbol\beta_0|{\bf d}}}{\langle U\rangle_{\boldsymbol\beta_0}}-1\right|>\epsilon\right]
    =
    \lim_{P\to\infty}
    {\rm Prob}\left[\left|\frac{\langle V\rangle_{\boldsymbol\beta_0|{\bf d}}}{\langle V\rangle_{\boldsymbol\beta_0}}-1\right|>\epsilon\right]
    =0.
    \label{eq:almostsureleadingorderstationarity}
\end{equation}
More symbolically, we have
    $\langle U\rangle_{\boldsymbol\beta_0|{\bf d}} =\langle U\rangle_{\boldsymbol\beta_0} +\mathcal O_p(\sqrt N)
    $, 
    $
    \langle V\rangle_{\boldsymbol\beta_0|{\bf d}} =\langle V\rangle_{\boldsymbol\beta_0} +\mathcal O_p(\sqrt P)
    $. 
We also obtain the following for fixed $P/N$:
\begin{equation}
    \lim_{N,P\to\infty}{\rm Prob}\left[\left|\frac 1 N\partial_{\boldsymbol\beta}\ln P({\bf d};\boldsymbol\beta)\right|_{\boldsymbol\beta=\boldsymbol\beta_0}>\epsilon\right]=0.
    \label{eq:almostsureleadingorderstationarity_lnPform}
\end{equation}
The essence in the derivation of eq.~(\ref{eq:almostsureleadingorderstationarity}) is the extensive property of the means and variances of the cost functions $U$ and $V$, which makes the randomness of the cost functions asymptotically negligible in their leading orders.

Since the true hyperparameters $\boldsymbol\beta_0$ almost surely satisfy the stationarity condition of the log marginal likelihood of the hyperparameters $\ln P ({\bf d};\boldsymbol\beta)$ in leading orders, there is almost surely a single local maximum of $P ({\bf d};\boldsymbol\beta)$ near $\boldsymbol\beta_0$ if $\ln P ({\bf d};\boldsymbol\beta)$ is locally sufficiently concave having the same-order curvature with probability 1 at $\boldsymbol\beta=\boldsymbol\beta_0$. 
We actually find such local sufficient concavity of $\ln P ({\bf d};\boldsymbol\beta)$ at $\boldsymbol\beta=\boldsymbol\beta_0$. 
Its derivation relies on a quadratic bound for the $2\times 2$ matrix ${\mathcal I}({\bf d},\boldsymbol\beta):=-\partial^2_{\boldsymbol\beta}\ln P ({\bf d};\boldsymbol\beta)$, which corresponds to the observed Fisher information matrix for $\boldsymbol\beta=\hat{\boldsymbol\beta}$; using a unit vector $\hat x$, we have
\begin{equation}
\begin{aligned}
    \hat x^{\rm T}\mathcal I\hat x 
    &
    \geq \frac{{\rm Tr}(\mathcal I)}2-\sqrt{\frac{{\rm Tr}^2(\mathcal I)}4-|\mathcal I|}
    \geq 
    \frac{{\rm Tr}(\mathcal I)}2-\left|\frac{{\rm Tr}(\mathcal I)}2-\frac{|\mathcal I|}{{\rm Tr}(\mathcal I)}\right|
    \\&
    \geq \frac{|\mathcal I|}{|{\rm Tr}(\mathcal I)|}+\frac{1}{2}[{\rm Tr}(\mathcal I)-|{\rm Tr}(\mathcal I)|],
\end{aligned}
\label{eq:eigenvaluebound}
\end{equation}
where ${\rm Tr}[\cdot]$ represents the trace for a matrix; 
the first transform represents the bound of the quadratic left hand side using the minimum eigenvalue of $\mathcal I$, and the second and third transforms utilize the properties of the quadratic and absolute functions with ${\rm Tr}^2(\mathcal I)/4\geq|\mathcal I|$. 
Inequality~(\ref{eq:eigenvaluebound}) bounds the minimum eigenvalue of a $2\times 2$ matrix by its determinant and trace. 
We use ineq.~(\ref{eq:eigenvaluebound}) and examine the positive definiteness of $\mathcal I$ in its leading order by evaluating the determinant and trace of $\mathcal I$ from the derivatives of $\ln P ({\bf d};\boldsymbol\beta)$:
\begin{equation}
\begin{aligned}
    &\left(\frac{\partial \ln P({\bf d};\boldsymbol\beta)}{\partial \beta_{\rm d}}\right)_{\beta_{\rm a}}=-U_*+ \frac{N-M}{2\beta_{\rm d}}+\frac{\lambda}{2\beta_{\rm d}}{\rm Tr}[{\bf J}]
\\
    &\left(\frac{\partial \ln P({\bf d};\boldsymbol\beta)}{\partial \beta_{\rm a}}\right)_{\beta_{\rm d}}=-V_*+ \frac{P}{2\beta_{\rm a}}-\frac{\lambda}{2\beta_{\rm a}}{\rm Tr}[{\bf J}]
\end{aligned}
\label{eq:partialfirstderiv_beta}
\end{equation}
and 
\begin{equation}
\begin{aligned}
    &\left(\frac{\partial^2 \ln P({\bf d};\boldsymbol\beta)}{\partial \beta_{\rm d}^2}\right)_{\beta_{\rm a}}=-\frac{N-M}{2\beta_{\rm d}^2}-\frac{\lambda}{2\beta_{\rm d}^2}{\rm Tr}[{\bf J}]
    -\lambda \left(\frac{\partial}{\partial\beta_{\rm d}}\left(\frac{\partial \ln P({\bf d};\boldsymbol\beta)}{\partial \beta_{\rm a}}\right)_{\beta_{\rm d}}\right)_{\beta_{\rm a}}
    \\
    &\left(\frac{\partial^2 \ln P({\bf d};\boldsymbol\beta)}{\partial \beta_{\rm a}^2}\right)_{\beta_{\rm d}}=-\frac{P}{2\beta_{\rm a}^2}+\frac{\lambda}{2\beta_a^2}{\rm Tr}[{\bf J}]
    -\frac 1 \lambda \left(\frac{\partial}{\partial\beta_{\rm d}}\left(\frac{\partial \ln P({\bf d};\boldsymbol\beta)}{\partial \beta_{\rm a}}\right)_{\beta_{\rm d}}\right)_{\beta_{\rm a}}
    \\
    &\left(\frac{\partial}{\partial\beta_{\rm d}}\left(\frac{\partial \ln P({\bf d};\boldsymbol\beta)}{\partial \beta_{\rm a}}\right)_{\beta_{\rm d}}\right)_{\beta_{\rm a}}
    =\frac{\lambda}{\beta_{\rm d}}\frac{d V_*}{d\lambda}+\frac{\lambda}{2\beta_{\rm a}\beta_{\rm d}}{\rm Tr}[{\bf J}({\bf I}-\lambda {\bf J})]
\end{aligned}
\label{eq:secondderiv}
\end{equation}
where ${\bf J}(\lambda):=({\bf H}^{\rm T}{\bf E}^{-1}{\bf H}+\lambda {\bf G})^{-1}{\bf G}$, ${\bf I}$ denotes a unit matrix, and subscripts of partial derivatives represent fixed variables in the differentiation. 
We utilize $d{\bf J}/d\lambda=-{\bf J}^2$ and the following relation:
\begin{equation}
\frac{dU_*(\lambda)}{d\lambda}+\lambda \frac{d V_*(\lambda)}{d\lambda}=0,
\end{equation}
which comes from the mode condition $d(U+\lambda V)/d{\bf a}={\bf 0}$ of the model-parameter posterior $P({\bf a}|{\bf d};\boldsymbol\beta)$ satisfied at ${\bf a}={\bf a}_*$. 
Equations~(\ref{eq:secondderiv}) show the randomness of the second derivatives are all represented by that of the cross derivative, or more specifically by $dV_*/d\lambda$. 
Then, utilizing $\ln P ({\bf d};\boldsymbol\beta)$ is a Gaussian weighted by 
a quadratic $U_*+\lambda V_*$, 
\begin{equation}
    U_*({\bf d},\lambda)+\lambda V_*({\bf d},\lambda)=\frac 1 2 {\bf d}^{\rm T}[{\bf E}^{-1}-{\bf E}^{-1}{\bf H}{\bf C}^\prime(\lambda){\bf H}^{\rm T}{\bf E}^{-1}]{\bf d},
    \label{eq:effecitiveposteriorweightfunction}
\end{equation}
where ${\bf C}^\prime(\lambda):=({\bf H}^{\rm T}{\bf E}^{-1}{\bf H}+\lambda {\bf G})^{-1}$ denotes the normalized covariance of $P({\bf a}|{\bf d};\boldsymbol\beta)$, 
and using $\langle({\bf d}^T...{\bf d})^2\rangle_{\boldsymbol\beta}=(\langle{\bf d}^T...{\bf d}\rangle)_{\boldsymbol\beta}^2+2\langle{\bf d}^T...\langle{\bf d}{\bf d}^T\rangle_{\boldsymbol\beta}...{\bf d}\rangle_{\boldsymbol\beta}$ for Gaussian ${\bf d}$, 
we evaluate the mean and variance for the cross derivative of $\ln P ({\bf d};\boldsymbol\beta)$: 
\begin{equation}
\begin{aligned}
    \left\langle\left(\frac{\partial}{\partial \beta_{\rm d}}\left(\frac{\partial \ln P({\bf d};\boldsymbol\beta)}{\partial \beta_{\rm a}}\right)_{\beta_{\rm d}}\right)_{\beta_{\rm a}}\right\rangle_{\boldsymbol\beta}
    =-\frac1{2\beta_{\rm d}\beta_{\rm a}}{\rm Tr}[{\bf H}{\bf C}^\prime{\bf H}^{\rm T}\tilde{\bf E}]
    \\
    {\rm Var}\left[\left(\frac{\partial}{\partial \beta_{\rm d}}\left(\frac{\partial \ln P({\bf d};\boldsymbol\beta)}{\partial \beta_{\rm a}}\right)_{\beta_{\rm d}}\right)_{\beta_{\rm a}};\boldsymbol\beta\right]
    =\frac{2}{\beta_{\rm d}^2\beta_{\rm a}^2}{\rm Tr}[({\bf H}{\bf C}^\prime{\bf H}^{\rm T}\tilde{\bf E})^2]
    \end{aligned}
    \label{eq:meanandvarianceofcrossderivativeatbetazero}
\end{equation}
where 
$\tilde{\bf E}(\lambda):={\bf E}^{-1}-{\bf E}^{-1}{\bf H}{\bf C}^\prime(\lambda){\bf H}^{\rm T}{\bf E}^{-1}$ is the matrix coefficient of $U_*+\lambda V_*$ (eq.~\ref{eq:effecitiveposteriorweightfunction}), corresponding to the generalized inverse of the normalized data covariance for $P ({\bf d};\boldsymbol\beta)$ in the limit of the improper prior $u\to0$. 
Note $\lambda{\bf HJC}^\prime{\bf H}^{\rm T}{\bf E}^{-1}={\bf HC}^\prime{\bf H}^{\rm T}\tilde{\bf E}$.
Equations~(\ref{eq:secondderiv}) and (\ref{eq:meanandvarianceofcrossderivativeatbetazero}) show the observed Fisher information almost surely reproduces the Fisher information in leading orders when $\hat{\boldsymbol\beta}=\boldsymbol\beta_0$: 
\begin{equation}
    \lim_{N,M,P\to\infty}{\rm Prob}\left[\left|\left|-\frac 1 {N}\partial_{\boldsymbol\beta}^2\ln P({\bf d};\boldsymbol\beta)-\left\langle-\frac 1 {N} \partial_{\boldsymbol\beta}^2\ln P({\bf d};\boldsymbol\beta)\right\rangle_{\boldsymbol\beta_0}\right|\right|_{\boldsymbol\beta=\boldsymbol\beta_0}>\epsilon\right]=0,
    \label{eq:almostsureconvergenceFisher}
\end{equation}
where $||\cdot||$ denotes the Frobenius norm for a matrix, and we evaluate ${\rm Tr}({\bf A})=\Theta[{\rm rk}({\bf A})]$ for a matrix ${\bf A}$; symbolically, 
$ 
    -\partial_{\boldsymbol\beta}^2\ln P({\bf d};\boldsymbol\beta)=\left\langle- \partial_{\boldsymbol\beta}^2\ln P({\bf d};\boldsymbol\beta)\right\rangle_{\boldsymbol\beta_0}+\mathcal O_p(\sqrt {N})+\mathcal O_p(\sqrt {M})
$ for $\boldsymbol\beta=\boldsymbol\beta_0$. 
We continue the matrix analysis of $\mathcal I$ on average given its convergence in probability: 
\begin{equation}
\begin{aligned}
    &\left\langle-\left(\frac{\partial^2 \ln P({\bf d};\boldsymbol\beta)}{\partial \beta_{\rm d}^2}\right)_{\beta_{\rm a}}\right\rangle_{\boldsymbol\beta}=
    \left\langle
    \frac{U_*}{\beta_{\rm d}}
    +\lambda \left(\frac{\partial}{\partial\beta_{\rm d}}\left(\frac{\partial \ln P({\bf d};\boldsymbol\beta)}{\partial \beta_{\rm a}}\right)_{\beta_{\rm d}}\right)_{\beta_{\rm a}}
    \right\rangle_{\boldsymbol\beta}
    \\
    &\left\langle-\left(\frac{\partial^2 \ln P({\bf d};\boldsymbol\beta)}{\partial \beta_{\rm a}^2}\right)_{\beta_{\rm d}}\right\rangle_{\boldsymbol\beta}=
    \left\langle\frac{V_*}{\beta_{\rm a}}
    +\frac 1 \lambda \left(\frac{\partial}{\partial\beta_{\rm d}}\left(\frac{\partial \ln P({\bf d};\boldsymbol\beta)}{\partial \beta_{\rm a}}\right)_{\beta_{\rm d}}\right)_{\beta_{\rm a}}
    \right\rangle_{\boldsymbol\beta}
\end{aligned}
\end{equation}
or explicitly, 
\begin{equation}
    \begin{aligned}
    &\left\langle-\left(\frac{\partial^2 \ln P({\bf d};\boldsymbol\beta)}{\partial \beta_{\rm d}^2}\right)_{\beta_{\rm a}}\right\rangle_{\boldsymbol\beta}=
    \frac{1}{2\beta_{{\rm d}}^2}{\rm Tr}[({\bf E}\tilde{\bf E})^2]
    \\
    &\left\langle-\left(\frac{\partial^2 \ln P({\bf d};\boldsymbol\beta)}{\partial \beta_{\rm a}^2}\right)_{\beta_{\rm d}}\right\rangle_{\boldsymbol\beta}=
    \frac{1}{2\beta_{{\rm a}}^2}{\rm Tr}[({\bf I}-{\bf E}\tilde{\bf E})^2\tilde{\bf E}\tilde{\bf E}^\dag]
    \end{aligned}
\end{equation}
where we use $\langle\partial_{\boldsymbol\beta}\ln P({\bf d};\boldsymbol\beta)\rangle_{\boldsymbol\beta}={\bf 0}$ (i.e. eq.~\ref{eq:avIbarel2_probform}) 
and the means of $U_*$ and $V_*$: 
\begin{equation}
\begin{aligned}
    \langle U_*({\bf d},\lambda)\rangle_{\boldsymbol\beta}&=\frac{1}{2\beta_{\rm d}}{\rm Tr}[{\bf E}\tilde{\bf E}]
    \\
    \langle U_*+\lambda V_*({\bf d},\lambda)\rangle_{\boldsymbol\beta}&=\frac{N+P-M}{2\beta_{\rm d}}
\end{aligned}
\end{equation}
Finally, we acquire the trace and determinant of the associated Fisher information matrix $\langle \mathcal I\rangle$: 
\begin{equation}
    {\rm Tr}[\langle -\partial_{\boldsymbol\beta}^2\ln P({\bf d};\boldsymbol\beta)\rangle_{\boldsymbol\beta}]=\frac{1}{2\beta_{{\rm d}}^2}{\rm Tr}[({\bf E}\tilde{\bf E})^2]+\frac{1}{2\beta_{{\rm a}}^2}{\rm Tr}[({\bf I}-{\bf E}\tilde{\bf E})^2\tilde{\bf E}\tilde{\bf E}^\dag]
    \label{eq:explicitrepresentation_trace}
\end{equation}
and
\begin{equation}
    \left|\langle -\partial_{\boldsymbol\beta}^2\ln P({\bf d};\boldsymbol\beta)\rangle_{\boldsymbol\beta}\right|=
    \left(\frac{N+P-M}{2\beta_{\rm d}\beta_{\rm a}}\right)^2 c_{\rm det}
    \label{eq:explicitrepresentation_determinant}
\end{equation}
with
\begin{equation}
    c_{\rm det}(\lambda):=\left\{\frac{{\rm Tr}[({\bf E}\tilde {\bf E})^2]}{N+P-M}-\left[\frac{{\rm Tr}({\bf E}\tilde {\bf E})}{N+P-M}\right]^2\right\}> 0,
    \label{eq:cdetfordeterminant}
\end{equation}
where $c_{\rm det}$ is the difference between the mean square and squared mean with respect to the eigenvalues of ${\bf E}\tilde {\bf E}$, generally nonnegative, and zero if and only if all eigenvalues of ${\bf E}\tilde {\bf E}$ take the same value (i.e. ${\bf E}\tilde{\bf E}\propto {\bf I}$), as noticed from Jensen's inequality; 
${\bf E}\tilde{\bf E}\propto {\bf I}$ is realized only if ${\bf G}\propto{\bf H}^{\rm T}{\bf E}^{-1}{\bf H}$ with $N=M=P$, and thus $c_{\rm det}$ is positive as far as in the case of ${\bf G}\not\propto{\bf H}^{\rm T}{\bf E}^{-1}{\bf H}$ we investigate. 
We note ${\rm rk}(\tilde{\bf E})=N+P-M$, derived from that $P ({\bf d};\boldsymbol\beta)$ is rank-deficient for $P<M$ without the uniform filter $u$ in the model-parameter prior. 
With 
the above expressions of the trace and determinant of $\langle\mathcal I({\bf d};\boldsymbol\beta)\rangle_{\boldsymbol\beta}$ (eqs.~\ref{eq:explicitrepresentation_trace} and \ref{eq:explicitrepresentation_determinant}),
ineq.~(\ref{eq:eigenvaluebound}) provides an almost sure lower bound of $\mathcal I({\bf d},\boldsymbol\beta)$ for $\boldsymbol\beta=\boldsymbol\beta_0$ under the leading-order convergence in probability of $\mathcal I({\bf d},\boldsymbol\beta_0)$ (eq.~\ref{eq:almostsureconvergenceFisher}); defining $\mathcal I_0({\bf d}):=\mathcal I({\bf d},\boldsymbol\beta_0)$ and $Nc_{\mathcal I}:=|\langle\mathcal I_0\rangle_{\boldsymbol\beta_0}|/{\rm Tr}(\langle\mathcal I_0\rangle_{\boldsymbol\beta_0})$,
ineq.~(\ref{eq:eigenvaluebound}) yields 
\begin{equation}
\begin{aligned}
    {\rm Prob}\left[\frac 1 N \hat x \mathcal I_0\hat x< c_{\mathcal I}-\epsilon\right]
    &\leq
    {\rm Prob}\left[\frac 1 N \hat x \mathcal I_0\hat x<\frac1N\hat x^{\rm T}\langle\mathcal I_0\rangle_{\boldsymbol\beta_0}\hat x-\epsilon\right]
\end{aligned}
\label{eq:rewritingI0boundwithFisherinfoconvergence}
\end{equation}
where we use ineq.~(\ref{eq:eigenvaluebound}) replacing $\mathcal I$ with the Fisher information matrix $\langle\mathcal I_0\rangle_{\boldsymbol\beta}$ of $P({\bf d};\boldsymbol\beta_0)$:
\begin{equation}
    \hat x^{\rm T}\langle\mathcal I_0\rangle_{\boldsymbol\beta_0}\hat x\geq Nc_{\mathcal I}.
    \label{eq:boundforFisherinfo}
\end{equation}
Applying Chebyshev's inequality with eq.~(\ref{eq:almostsureconvergenceFisher}), 
eq.~(\ref{eq:rewritingI0boundwithFisherinfoconvergence}) leads to
\begin{equation}
    \lim_{N,M,P\to\infty} {\rm Prob}\left[\frac 1 N \hat x \mathcal I_0\hat x< c_{\mathcal I}-\epsilon\right]=0,
\label{eq:almostsuresufficientconvexity}
\end{equation}
where we evaluate ${\rm Tr}({\bf A})$ as $\Theta[{\rm rk}({\bf A})]$ for a matrix ${\bf A}$. 
Equation~(\ref{eq:almostsuresufficientconvexity}) indicates $P({\bf d};\boldsymbol\beta)$ is in its leading order asymptotically almost surely locally concave at $\boldsymbol\beta=\boldsymbol\beta_0$ when the given data ${\bf d}$ is generated by the prior predictive distribution $P({\bf d};\boldsymbol\beta)$ of $\boldsymbol\beta=\boldsymbol\beta_0$. 
Given eq.~(\ref{eq:secondderivlnP}), eq.~(\ref{eq:almostsuresufficientconvexity}) also means the almost sure posterior covariance reduction with respect to the cost functions $U$ and $V$ in asymptotics: 
for the covariance matrices of the cost function vector $\boldsymbol{\mathcal E}$, 
\begin{equation}
    \lim_{N,M,P\to\infty}{\rm Prob}\left[\frac 1 N {\rm Cov}[\boldsymbol{\mathcal E},\boldsymbol{\mathcal E}^{\rm T}|{\bf d};\boldsymbol\beta_0]\preceq \frac 1 N {\rm Cov}[\boldsymbol{\mathcal E},\boldsymbol{\mathcal E}^{\rm T};\boldsymbol\beta_0]- (c_{\mathcal I}-\epsilon) \hat x\hat x^{\rm T}\right]=1,
    \label{eq:posteriorcovariancereduction}
\end{equation}
where $\preceq$ denotes the positive semidefiniteness of the right-hand side minus the left-hand side. 
Equation~(\ref{eq:posteriorcovariancereduction}) comprehends 
asymptotic almost sure posterior variance reductions for $U$ and $V$:
\begin{equation}
\begin{aligned}
    &\lim_{N,M,P\to\infty}{\rm Prob}\left[\frac 1 N {\rm Var}[U|{\bf d};\boldsymbol\beta_0]\leq \frac 1 N {\rm Var}[U;\boldsymbol\beta_0]-(c_{\mathcal I}-\epsilon) \right]=1
    \\
    &\lim_{N,M,P\to\infty}{\rm Prob}\left[\frac 1 P {\rm Var}[V|{\bf d};\boldsymbol\beta_0]\leq \frac 1 P {\rm Var}[V;\boldsymbol\beta_0]-(c_{\mathcal I}-\epsilon) \right]=1
\end{aligned}
\end{equation}

The asymptotic almost sure leading-order stationarity and concavity (eqs.~\ref{eq:almostsureleadingorderstationarity_lnPform} and \ref{eq:almostsuresufficientconvexity}) of the log marginal likelihood of the hyperparameters $\ln P ({\bf d};\boldsymbol\beta)$ at the true hyperparameter set $\boldsymbol\beta=\boldsymbol\beta_0$ lead to a single stationary point of $\ln P ({\bf d};\boldsymbol\beta)$ in the $\epsilon$-neighborhood of $\boldsymbol\beta_0$ asymptotically with probability 1.
\begin{theorem}[Existence of the stationary point near the true hyperparameters]\label{th:localsufficientconvexity}
If exists, the true hyperparameter vector $\boldsymbol\beta_0$ is 
located within the $\epsilon$-neighborhood of one single stationary point in $\ln P({\bf d};\boldsymbol\beta)$ 
for the two-hyperparameter general linear models (eqs.~\ref{eq:twohyperparameterform})  with probability 1 in the limit $N,M\to\infty$ of $N/M=const.$, as long as the associated regularized least square estimate is well-posed. 
\end{theorem}

\begin{proof}[Proof of Theorem~\ref{th:localsufficientconvexity}]
The following corollary of eq.~(\ref{eq:almostsureleadingorderstationarity_lnPform}) with eq.~(\ref{eq:almostsuresufficientconvexity}) is useful to evaluate the gradient variation of $\ln P({\bf d};\boldsymbol\beta)$ 
around $\boldsymbol\beta_0$ within an $\epsilon_\beta$-interval along any travelling direction $\hat x$: 
\begin{equation}
    \lim_{N,M,P\to\infty}{\rm Prob}\left[\pm \hat x^{\rm T}\partial_{\boldsymbol\beta}\ln P ({\bf d};\boldsymbol\beta_0\pm \epsilon_\beta\hat x)> 0\right]=0,
\label{eq:gradientbound}
\end{equation} 
which is obtained as follows [where we simplify $P ({\bf d};\boldsymbol\beta)$ as $L(\boldsymbol\beta)$]:
\begin{equation}
\begin{aligned}
    &
    {\rm Prob}\left[\pm \hat x^{\rm T}\partial_{\boldsymbol\beta}\ln L (\boldsymbol\beta_0\pm \epsilon_\beta\hat x)> 0\right]
    \\=&
    {\rm Prob}\left[\pm \frac 1 N \hat x^{\rm T}\partial_{\boldsymbol\beta}\ln L (\boldsymbol\beta_0)
    -\frac {c_{\rm TS}\epsilon_\beta} N 
    \hat x^{\rm T}\mathcal I_0\hat x
    > 0\right]
    \\\leq &
    {\rm Prob}\left[\pm \frac 1 N \hat x^{\rm T}\partial_{\boldsymbol\beta}\ln L (\boldsymbol\beta_0)
    -c_{\rm TS}\epsilon_\beta(c_{\mathcal I}-\epsilon)
    > 0\right]
    +
    {\rm Prob}\left[\frac 1 N 
    \hat x^{\rm T}\mathcal I_0\hat x<c_{\mathcal I}-\epsilon
    \right]
    \\\leq &
    {\rm Prob}\left[\left|\frac 1 N \hat x^{\rm T}\partial_{\boldsymbol\beta}\ln L (\boldsymbol\beta_0)\right|
    > c_{\rm TS}\epsilon_\beta(c_{\mathcal I}-\epsilon)\right]
    +
    {\rm Prob}\left[\frac 1 N 
    \hat x^{\rm T}\mathcal I_0\hat x<c_{\mathcal I}-\epsilon
    \right]
    \\\to&0 \hspace{4pt}(N,M,P\to\infty),
\end{aligned}
\label{eq:boundinglocalgradient}
\end{equation}
where the first transform employs a Taylor series with defining an $\mathcal O(1)$ coefficient,
\begin{equation}
    c_{\rm TS}:=
    \left[\epsilon_\beta \hat x^{\rm T}\frac{\partial^2\ln P}{\partial\boldsymbol\beta\partial\boldsymbol\beta^{\rm T}} ({\bf d};\boldsymbol\beta_0)\hat x\right]^{-1}
    \int^{\epsilon_\beta}_0d\delta\beta^\prime
    \hat x^{\rm T}\frac{\partial^2\ln P}{\partial\boldsymbol\beta\partial\boldsymbol\beta^{\rm T}} ({\bf d};\boldsymbol\beta_0+\delta\beta^\prime\hat x)\hat x=1+\mathcal O(\epsilon_\beta),
\end{equation}
and the left transforms utilize eqs.~(\ref{eq:almostsureleadingorderstationarity_lnPform}) and (\ref{eq:almostsuresufficientconvexity}) using a small positive number $\epsilon$. 
The symmetric quadratic of the Fisher information matrix plays a role of the absolute normal curvature on the log probability landscape in eq.~(\ref{eq:boundinglocalgradient}) as in information geometry~\citep{amari2000methods,fisher1922mathematical}. 
Equation~(\ref{eq:gradientbound}) immediately concludes a stationary point of the marginal likelihood of the hyperparameters almost surely exists within the $\epsilon$-neighborhood of $\boldsymbol\beta_0$: using unit vectors $\hat x_{\rm d}$ and $\hat x_{\rm a}$ respectively parallel to the $\beta_{\rm d}$ and $\beta_{\rm a}$ axes, 
\begin{equation}
\begin{aligned}
    &
    {\rm Prob}\left[\bigcap_{\hat x=\hat x_{\rm d},\hat x_{\rm a}} [\left( \hat x^{\rm T}\partial_{\boldsymbol\beta}\ln P ({\bf d};\boldsymbol\beta_0- \epsilon_\beta\hat x)\geq 0\right)
    \cap
    \left(\hat x^{\rm T}\partial_{\boldsymbol\beta}\ln P ({\bf d};\boldsymbol\beta_0+ \epsilon_\beta\hat x)\leq 0\right)]\right]
    \\&\geq 1 -
    \sum_{\pm=\pm1}\sum_{\hat x=\hat x_{\rm d},\hat x_{\rm a}}
    {\rm Prob}\left[\pm \hat x^{\rm T}\partial_{\boldsymbol\beta}\ln P ({\bf d};\boldsymbol\beta_0\pm \epsilon_\beta\hat x)> 0\right]
    \to1 \hspace{4pt}(N,M,P\to\infty).
\end{aligned}
\label{eq:stationarityproofbody}
\end{equation}
Considering $\epsilon\geq\sqrt{2}\epsilon_\beta$ with $\epsilon_{\beta}<\beta_{d0}, \beta_{a0}$, 
(in)eq.~(\ref{eq:stationarityproofbody}) assures the almost sure existence of at least one stationary point within the $\epsilon$-neighborhood of $\boldsymbol\beta_0$ for the marginal likelihood of the hyperparameters. 
The oneness of the stationary point also follows by taking a sufficiently small $\epsilon$ compared to the characteristic hyperparameter scale of the $\ln P ({\bf d}; \boldsymbol\beta)$ variation; and even if not, the multiple stationary points effectively merge in the limit of infinitesimal $\epsilon$. 
Symbolically, there is asymptotically almost surely an $\mathcal O(1/\sqrt{N})$ vector $\boldsymbol\epsilon_*$ such that
\begin{equation}
    \left.\frac{\partial \ln P({\bf d};\boldsymbol\beta)}{\partial \boldsymbol\beta}\right|_{\boldsymbol\beta=\boldsymbol\beta_0+\boldsymbol\epsilon_*}=0.
\end{equation}
\end{proof}

\subsection{Global analysis: A complete proof}
\label{sec:proofwithoutassuminguniquestationarity}

Now that we saw there is asymptotically almost surely a local maximum of the marginal likelihood of the hyperparameters $P ({\bf d};\boldsymbol\beta)$ within the $\epsilon$-neighborhood of the true hyperparameter set $\boldsymbol\beta=\boldsymbol\beta_0$, 
we last show it is asymptotically almost surely the mode of $P ({\bf d};\boldsymbol\beta)$. 
The associated relation is derived with the following strategy. 
We first show the leading-order convergence in probability of the empirical cross entropy $-\ln P ({\bf d};\boldsymbol\beta)$ of $P ({\bf d};\boldsymbol\beta)$ relative to $P ({\bf d};\boldsymbol\beta_0)$ for a single data sample ${\bf d}$, besides those of its first (and second) derivatives. 
Next we find the zero point of $\partial_{\boldsymbol\beta}\langle\ln P ({\bf d};\boldsymbol\beta)\rangle_{\boldsymbol\beta_0}=\langle\partial_{\boldsymbol\beta}\ln P ({\bf d};\boldsymbol\beta)\rangle_{\boldsymbol\beta_0}$ is uniquely the mode $\boldsymbol\beta=\boldsymbol\beta_0$ of $\langle\ln P ({\bf d};\boldsymbol\beta)\rangle_{\boldsymbol\beta_0}$, that is the uniqueness of the stationary point in the cross entropy (and the Kullback-Leibler divergence) of $P ({\bf d};\boldsymbol\beta)$ relative to $P ({\bf d};\boldsymbol\beta_0)$. 
From these and the aforementioned concavity of $\ln P ({\bf d};\boldsymbol\beta)$ at $\boldsymbol\beta=\boldsymbol\beta_0$ in probability asymptotics, we derive the asymptotic almost sure modality of $\boldsymbol\beta=\boldsymbol\beta_0+\mathcal O(\epsilon)$ in $\ln P ({\bf d};\boldsymbol\beta)$. 

Because $P({\bf d};\boldsymbol\beta)$ for fixed $\lambda$ is a Gaussian formally identical to $P({\bf d}|{\bf a};\beta_{\rm d})$ (mapped if ${\bf E}^{-1}\to\tilde{\bf E}$) being concave with respect to $\beta_{\rm d}$, 
the nontrivial part of our mode search falls within the one-dimensional $\lambda$ domain. 
The $\boldsymbol\beta^\prime:=(\beta_{\rm d},\lambda)^{\rm T}$ space provides the following partial derivatives of $\ln P ({\bf d};\boldsymbol\beta)$: 
\begin{equation}
    \begin{aligned}
&\left(\frac{\partial\ln P({\bf d};\boldsymbol\beta)}{\partial \beta_{\rm d}}\right)_\lambda=-(U_*+\lambda V_*)+\frac{N+P-M}{2\beta_{\rm d}}
\\
&\left(\frac{\partial\ln P({\bf d};\boldsymbol\beta)}{\partial \lambda}\right)_{\beta_{\rm d}}=-\beta_{\rm d}V_*+\frac{P}{2\lambda}-\frac 1 2 {\rm Tr}[{\bf J}]
    \end{aligned}
    \label{eq:partialfirstderiv_betaprime}
\end{equation}
and 
\begin{equation}
    \begin{aligned}
&\left(\frac{\partial^2\ln P({\bf d};\boldsymbol\beta)}{\partial \beta_{\rm d}^2}\right)_\lambda=-\frac{N+P-M}{2\beta_{\rm d}^2}
\\
&\left(\frac{\partial^2\ln P({\bf d};\boldsymbol\beta)}{\partial \lambda^2}\right)_{\beta_{\rm d}}=-\beta_{\rm d}\frac{d V_*}{d\lambda}-\frac{P}{2\lambda^2}+\frac 1 2 {\rm Tr}[{\bf J}^2]
\\
&\left(\frac{\partial}{\partial\lambda} \left(\frac{\partial\ln P({\bf d};\boldsymbol\beta)}{\partial \beta_{\rm d}}\right)_\lambda\right)_{\beta_{\rm d}}=-V_*
    \end{aligned}
    \label{eq:secondderivinbetaprimespace}
\end{equation}
The random variables in the first derivatives of $\ln P({\bf d};\boldsymbol\beta)$ are limited to $U_*$ and $V_*$ (eqs.~\ref{eq:partialfirstderiv_beta} and \ref{eq:partialfirstderiv_betaprime} for $\boldsymbol\beta$ and $\boldsymbol\beta^\prime$, respectively), which have the following means and variances: 
\begin{equation}
\begin{aligned}
    \langle U_*({\bf d},\lambda)\rangle_{\boldsymbol\beta_0} &= \frac{1}{2\beta_{\rm d0}}{\rm Tr}[{\bf E}\tilde {\bf E}^2\tilde {\bf E}^\dag_0]
    \\
    \langle \lambda V_*({\bf d},\lambda)\rangle_{\boldsymbol\beta_0} &=\frac{1}{2\beta_{\rm d0}}{\rm Tr}[({\bf I}-{\bf E}\tilde {\bf E})\tilde {\bf E}\tilde {\bf E}^\dag_0]
    \\
    {\rm Var}[U_*({\bf d},\lambda);\boldsymbol\beta_0] &=\frac{1}{2\beta_{\rm d0}^2}{\rm Tr}[({\bf E}\tilde {\bf E}^2\tilde {\bf E}^\dag_0)^2]
    \\
    {\rm Var}[\lambda V_*({\bf d},\lambda);\boldsymbol\beta_0] &=\frac{1}{2\beta_{\rm d0}^2}{\rm Tr}\{[({\bf I}-{\bf E}\tilde {\bf E})\tilde {\bf E}\tilde {\bf E}^\dag_0]^2\}
\end{aligned}
\label{eq:UstarVstarmeanvariance}
\end{equation}
where $\tilde{\bf E}_0:=\tilde{\bf E}(\lambda_0)$. Note the null eigenspace of $\tilde {\bf E}$ is independent of the $\lambda$ value. 
Equations~(\ref{eq:UstarVstarmeanvariance}) with Chebyshev's inequality indicate $U_*$ and $V_*$ are convergent in probability in their leading orders: 
\begin{equation}
    \begin{aligned}
    \lim_{N,P,M\to\infty}
    {\rm Prob}\left[\left|\frac{U_*}{\langle U_*\rangle_{\boldsymbol\beta_0}}-1\right|>\epsilon\right]
    =
    \lim_{N,P,M\to\infty}
    {\rm Prob}\left[\left|\frac{V_*}{\langle  V_*\rangle_{\boldsymbol\beta_0}}-1\right|>\epsilon\right]
    =0,
    \end{aligned}
\end{equation}
where we evaluate ${\rm Tr}({\bf A})$ as $\Theta[{\rm rk}({\bf A})]$ for a matrix ${\bf A}$. 
Then, the gradient of $\ln P({\bf d};\boldsymbol\beta)$ is found convergent in its leading order both in the $\boldsymbol\beta$ and $\boldsymbol\beta^\prime$ space: 
\begin{equation}
\lim_{N,P,M\to\infty}
    {\rm Prob}\left[
    \frac{1}{N}\left|\partial_{\boldsymbol\beta}\ln P({\bf d};\boldsymbol\beta)-
    \langle
    \partial_{\boldsymbol\beta}\ln P({\bf d};\boldsymbol\beta)
    \rangle_{\boldsymbol\beta_0}
    \right|
    >\epsilon
    \right]=0.
    \label{eq:deterministiccrossentropygradient}
\end{equation}
We note the following transform assures one-to-one correspondence between the gradient of $\ln P({\bf d};\boldsymbol\beta)$ in the $\boldsymbol\beta$ and $\boldsymbol\beta^\prime$ spaces: 
\begin{equation}
    \partial_{\boldsymbol\beta}\ln P({\bf d};\boldsymbol\beta)=
    \left(
    \begin{array}{cc}
         1&  -\lambda/\beta_{\rm d}\\
         0& 1/\beta_{\rm d}
    \end{array}
    \right)
    \partial_{\boldsymbol\beta^\prime}\ln P({\bf d};\boldsymbol\beta).
\end{equation}
Further recalling eq.~(\ref{eq:priorpredictivemodel}), we find $\ln P({\bf d};\boldsymbol\beta)$ is also convergent in probability in the leading order:
\begin{equation}
    \lim_{N,P,M\to\infty}
    {\rm Prob}\left[
    \frac{1}{N+P-M}\left|-\ln P({\bf d};\boldsymbol\beta)-
    \langle
    -\ln P({\bf d};\boldsymbol\beta)
    \rangle_{\boldsymbol\beta_0}
    \right|
    >\epsilon
    \right]=0,
    \label{eq:deterministiccrossentropy}
\end{equation}
where the mean and variance of $U_*+\lambda V_*$ are both $ \mathcal O[{\rm rk}({\bf E}\tilde{\bf E})]=\mathcal O\{{\rm rk}[({\bf E}\tilde{\bf E})^2]\}=\mathcal O(N+P-M)$. 
Symbolically,
$
    -\ln P({\bf d};\boldsymbol\beta)=\langle
    -\ln P({\bf d};\boldsymbol\beta)
    \rangle_{\boldsymbol\beta_0}+\mathcal O_p(\sqrt{N+P-M})
$. 
Equation~(\ref{eq:deterministiccrossentropy}) implies leading-order determinicacy (convergence in probability) of the empirical cross entropy of $P({\bf d};\boldsymbol\beta)$ relative to $P({\bf d};\boldsymbol\beta_0)$ even for one data sample ${\bf d}$,
and the same for their Kullback-Leibler divergence:
\begin{equation}
    \lim_{N,P,M\to\infty}
    {\rm Prob}\left[
    \frac{1}{N+P-M}\left|-\ln\frac{P({\bf d};\boldsymbol\beta)}{P({\bf d};\boldsymbol\beta_0)}-
    \mathcal D[P({\bf d};\boldsymbol\beta)||P({\bf d};\boldsymbol\beta_0)]
    \right|
    >\epsilon
    \right]=0.
\end{equation}
The leading-order convergence in probability applies also to the second derivatives of $\ln P({\bf d};\boldsymbol\beta)$, although $\langle \ln P({\bf d};\boldsymbol\beta)\rangle_{\boldsymbol\beta_0}$ is not the Fisher information (generally positive semidefinite) and thus can be negative definite, except when $\boldsymbol\beta=\boldsymbol\beta_0$ (eq.~\ref{eq:almostsureconvergenceFisher}). 
Actually, the uniqueness of the stationary point in nonconvex $\langle- \ln P({\bf d};\boldsymbol\beta)\rangle_{\boldsymbol\beta_0}$ is a central subject in the following analysis. 
The above convergence in probability of $\ln P({\bf d};\boldsymbol\beta)$ and those of its derivatives intrinsically jointly follow given the permutability between the partial differentiation by $\boldsymbol\beta$ and averaging by the true data distribution $P({\bf d};\boldsymbol\beta_0)$ (i.e. $\partial_{\boldsymbol\beta}\langle\cdot\rangle_{\boldsymbol\beta_0}=\langle\partial_{\boldsymbol\beta}\cdot\rangle_{\boldsymbol\beta_0}$), 
although the calculation complicates as the leading order of each differentiated $\ln P({\bf d};\boldsymbol\beta)$ is a mixture of $N, M,P$-dependent parts as seen earlier. 

The above relations for the convergence in probability treat a single $\boldsymbol\beta$ value, rather than multiple $\boldsymbol\beta$ for the functional convergence. Yet, we can show the values of $\ln P({\bf d};\boldsymbol\beta)$ (as well as $\partial_{\boldsymbol\beta}\ln P({\bf d};\boldsymbol\beta)$ and $\partial_{\boldsymbol\beta}^2\ln P({\bf d};\boldsymbol\beta)$) for a finite number of $\boldsymbol\beta$ values $\{\boldsymbol\beta_n\}_{n=1,2,...}$ simultaneously satisfy the convergence in probability, considering that the probability of $\epsilon$-scattering at any of $\boldsymbol\beta$ values (${\rm Prob}[(|...|>\epsilon)\cup...]$) does not exceed the sum of the $\epsilon$-scattering probability over each of $\boldsymbol\beta$ values. 
Therefore, in its limit with sufficiently finer grids than the convergence radii of the Taylor series of $\ln P({\bf d};\boldsymbol\beta)$ within a sufficiently wide hyperparameter subdomain, we notice $\ln P({\bf d};\boldsymbol\beta)$ observes the convergence in probability even as a function.

The leading-order convergence in probability of $\partial_{\boldsymbol\beta}\ln P({\bf d};\boldsymbol\beta)$ 
implies
the primal profile of $\partial_{\boldsymbol\beta}\ln P({\bf d};\boldsymbol\beta)$ is drawn by $\langle \partial_{\boldsymbol\beta}\ln P({\bf d};\boldsymbol\beta)\rangle_{\boldsymbol\beta_0}= \partial_{\boldsymbol\beta}\langle\ln P({\bf d};\boldsymbol\beta)\rangle_{\boldsymbol\beta_0}$ (the negative gradient of the cross entropy and the Kullback-Leibler divergence), where remarkably we find the uniqueness of the zero point for this model setting as below. 
Since $\ln P({\bf d};\boldsymbol\beta)$ is always monotonic for fixed $\lambda$, we first partially maximize $\langle \partial_{\boldsymbol\beta}\ln P({\bf d};\boldsymbol\beta)\rangle_{\boldsymbol\beta_0}$ with respect to $\beta_{\rm d}$,
analogously to the profile likelihood method~\citep{murphy2000profile}. 
From the zero-point condition of $\partial_{\beta_{\rm d}}\langle \ln P({\bf d};\boldsymbol\beta)\rangle_{\boldsymbol\beta_0}=\langle \partial_{\beta_{\rm d}}\ln P({\bf d};\boldsymbol\beta)\rangle_{\boldsymbol\beta_0}$ with using eqs.~(\ref{eq:partialfirstderiv_betaprime}) and (\ref{eq:UstarVstarmeanvariance}), we have an expression of the conditional mode $\tilde \beta_{\rm d}(\lambda)$ with respect to $\beta_{\rm d}$ given $\lambda$:
\begin{equation}
    \tilde \beta_{\rm d}(\lambda)=\left(\frac{{\rm Tr}(\tilde{\bf  E}\tilde{\bf  E}_0^\dag)}{N+P-M}\right)^{-1}\beta_{\rm d0},
    \label{eq:implicativebetadform}
\end{equation}
which is a monotinically decreasing function of $\lambda$. 
Substituting $\beta_{\rm d}=\tilde \beta_{\rm d}$ into $\partial_\lambda\langle\ln P({\bf d};\boldsymbol\beta)\rangle_{\boldsymbol\beta_0}=\langle\partial_\lambda\ln P({\bf d};\boldsymbol\beta)\rangle_{\boldsymbol\beta_0}$, and defining ${\bf D}(\lambda):={\bf E}\tilde{\bf E}(\lambda)$ and $N^\prime=N+P-M$, we have
\begin{equation}
    \partial_\lambda\langle\ln P({\bf d};\tilde \beta_{\rm d},\lambda)\rangle_{\boldsymbol\beta_0}=\frac{N^\prime}{2\lambda}\frac{{\rm Tr}[{\bf D}^2{\bf D}_0^\dag]-{\rm Tr}[{\bf D}{\bf D}_0^\dag]{\rm Tr}[{\bf D}]/N^\prime}{{\rm Tr}[{\bf D}{\bf D}_0^\dag]},
    \label{eq:differentiatedprofilelikelihoodlikefunc}
\end{equation}
where ${\bf D}_0={\bf D}(\lambda_0)$. 
We use ${\rm Tr}[({\bf D D}^\dag{\bf H}-{\bf HGG}^\dag){\bf C}^\prime{\bf H}^{\rm T}{\bf E}^{-1}]=0$, which follows $\partial_\lambda\langle\ln P({\bf d};\tilde \beta_{\rm d},\lambda)\rangle_{\boldsymbol\beta_0}=0$ at $\lambda=\lambda_0$. 
The zero points of eq.~(\ref{eq:differentiatedprofilelikelihoodlikefunc}) and associated $\tilde\beta_{\rm d}$ values coincide with the zero points of $ \partial_{\boldsymbol\beta}\langle\ln P({\bf d};\boldsymbol\beta)\rangle_{\boldsymbol\beta_0}$, and actually the true hyperparameter $\lambda=\lambda_0$ is one solution of $ \partial_{\boldsymbol\beta}\langle\ln P({\bf d};\boldsymbol\beta)\rangle_{\boldsymbol\beta_0}={\bf 0}$ with $\tilde\beta_{\rm d}(\lambda_0)=\beta_{\rm d0}$. 
Since it is hard to fully solve eq.~(\ref{eq:differentiatedprofilelikelihoodlikefunc}) with respect to $\lambda$, 
we further differentiate eq.~(\ref{eq:differentiatedprofilelikelihoodlikefunc}) with respect to $\lambda_0$ and widen the solution space of $\partial_{\boldsymbol\beta}\langle\ln P({\bf d};\tilde \beta_{\rm d},\lambda)\rangle_{\boldsymbol\beta_0}={\bf 0}$ also to the $\lambda_0$ axis: 
\begin{equation}
    \begin{aligned}
    \partial_{\lambda_0}\partial_\lambda\langle\ln P({\bf d};\tilde \beta_{\rm d},\lambda)\rangle_{\boldsymbol\beta_0}=&
    \{-1+{\rm Tr}[{\bf D}({\bf D}_0^\dag-{\bf I})]/{\rm Tr}[{\bf D}{\bf D}_0^\dag]\}\partial_\lambda\ln P({\bf d};\tilde \beta_{\rm d},\lambda)/\lambda_0
    \\&+
    \frac{N^\prime}{2\lambda\lambda_0}\frac{{\rm Tr}[{\bf D}^2]/N^\prime-({\rm Tr}[{\bf D}]/N^\prime)^2}{{\rm Tr}[{\bf D}{\bf D}_0^\dag]/N^\prime}.
    \end{aligned}
\label{eq:pdfdiffedbylambda0}
\end{equation}
Note $d{\bf D}/d\lambda={\bf D}({\bf I}-{\bf D})/\lambda$.
Especially for the zero points $\lambda_*(\lambda_0)$ of $\partial_\lambda\langle\ln P({\bf d};\tilde \beta_{\rm d},\lambda)\rangle_{\boldsymbol\beta_0}$, such that
\begin{equation}
    \left.\partial_\lambda\langle\ln P({\bf d};\tilde \beta_{\rm d},\lambda)\rangle_{\boldsymbol\beta_0}\right|_{\lambda=\lambda_*(\lambda_0)}=0,
\end{equation}
eq.~(\ref{eq:pdfdiffedbylambda0}) is simplified to
\begin{equation}
    \left.\partial_{\lambda_0}\partial_\lambda\langle\ln P({\bf d};\tilde \beta_{\rm d},\lambda)\rangle_{\boldsymbol\beta_0}\right|_{\lambda=\lambda_*(\lambda_0)}=
    \left.
    \frac{N^\prime}{2\lambda\lambda_0}\frac{{\rm Tr}[{\bf D}^2]/N^\prime-({\rm Tr}[{\bf D}]/N^\prime)^2}{{\rm Tr}[{\bf D}{\bf D}_0^\dag]/N^\prime}\right|_{\lambda=\lambda_*(\lambda_0)}\geq 0.
    \label{eq:floweqforlambdastar}
\end{equation}
The numerator of eq.~(\ref{eq:floweqforlambdastar}) is proportional to the normalized determinant of the Fisher information matrix ($c_{\rm det}>0$; eq.~\ref{eq:cdetfordeterminant}), and the right-hand side of eq.~(\ref{eq:floweqforlambdastar}) is found nonnegative for $\lambda\leq \infty$ (positive for $\lambda< \infty$). 
Referring to eq.~(\ref{eq:floweqforlambdastar}), we trace the evolution of $\lambda_*$ as a function of the true hyperparameter $\lambda_0$. 
The $\lambda_*$ evolution begins with the initial condition for sufficiently large $\lambda_0=\lambda_{0\rm init}\gg1$. 
The right-hand side of eq.~(\ref{eq:differentiatedprofilelikelihoodlikefunc}) for large $\lambda_0$ does not possess any zero points within a moderate $\lambda$ range; for $\lambda=\mathcal O(1)$,
\begin{equation}
    \partial_\lambda\langle\ln P({\bf d};\tilde \beta_{\rm d},\lambda)\rangle_{\boldsymbol\beta_0}=
    \frac{N^\prime}{2\lambda}\frac{{\rm Tr}[{\bf D}^2]/N^\prime-({\rm Tr}[{\bf D}]/N^\prime)^2}{{\rm Tr}[{\bf D}]/N^\prime}+\mathcal O(1/\lambda_{0{\rm init}})>0. 
    \label{eq:expandedgradientformoderatelambda}
\end{equation}
Its possible zero point is then limited to $\lambda\gg1$, where eq.~(\ref{eq:differentiatedprofilelikelihoodlikefunc}) is reduced to 
\begin{equation}
    \partial_\lambda\langle\ln P({\bf d};\tilde \beta_{\rm d},\lambda)\rangle_{\boldsymbol\beta_0}=
    -\frac{N^\prime(\lambda_{0{\rm init}}^{-1}-\lambda^{-1})}{2\lambda^2}c_{\lambda\infty}+\mathcal O[(1/\lambda+1/\lambda_{0{\rm init}})^3/\lambda]
    \label{eq:expandedgradientforlargelambda}
\end{equation}
with 
\begin{equation}
    c_{\lambda\infty}:=\left(\frac{1}{N^\prime}{\rm Tr}[({\bf G}^\dag{\bf H}^{\rm T}{\bf E}^{-1}{\bf H})^2]-\left[\frac{1}{N^\prime}{\rm Tr}({\bf G}^\dag{\bf H}^{\rm T}{\bf E}^{-1}{\bf H})\right]^2\right)> 0,
\end{equation}
which is positive for $N+P>M$ when ${\bf G}\not\propto{\bf H}^{\rm T}{\bf E}^{-1}{\bf H}$ as $c_{\rm det}$ (see below eq.~\ref{eq:cdetfordeterminant}; one may use the aforementioned relations: ${\rm Tr}[({\bf D D}^\dag{\bf H}-{\bf HGG}^\dag){\bf C}^\prime{\bf H}^{\rm T}{\bf E}^{-1}]=0$ and $\partial_\lambda({\bf D}{\bf D}^\dag)=0$). 
Equation~(\ref{eq:expandedgradientforlargelambda}) shows $\lambda_*(\lambda_{0{\rm init}})\gg1$ can exist only around $\lambda=\lambda_{0{\rm init}}\gg1$ and $\lambda=\infty$. From eq.~(\ref{eq:differentiatedprofilelikelihoodlikefunc}), we find the associated zero points are 
\begin{equation}
    \lambda_*(\lambda_{0{\rm init}})=\lambda_{0{\rm init}}, \infty.
    \label{eq:lambdastarforunshiftedlambda0}
\end{equation} 
Next, we solve an initial value problem of calculating $\lambda_*$ for $\lambda_0<\lambda_{0{\rm init}}$ in a recursive manner with eq.~(\ref{eq:floweqforlambdastar}). 
A Taylor series of eq.~(\ref{eq:differentiatedprofilelikelihoodlikefunc}) using eq.~(\ref{eq:floweqforlambdastar}) near $\lambda=\lambda_*(\lambda_{0{\rm init}})$ concludes 
the zero points $\lambda_*$ for $\epsilon_\lambda$-decreasing $\lambda_0(=\lambda_{0{\rm init}}-\epsilon_\lambda)$ almost satisfy the zero-point condition when $\lambda_0=\lambda_{0{\rm init}}$:
\begin{equation}
    \left.\partial_\lambda\langle\ln P({\bf d};\tilde \beta_{\rm d},\lambda)\rangle_{\boldsymbol\beta_0}\right|_{\lambda=\lambda_*(\lambda_{0{\rm init}}-\epsilon_\lambda),\lambda_0=\lambda_{0{\rm init}}}=\mathcal O(\epsilon_\lambda).
    \label{eq:almoststationaryrecursion}
\end{equation}
Because $\partial_\lambda\langle\ln P({\bf d};\tilde \beta_{\rm d},\lambda)\rangle_{\boldsymbol\beta_0}$ is a smooth function of $\lambda_0$, 
eq.~(\ref{eq:almoststationaryrecursion}) means 
the evolved $1/\lambda_*$ keeps within the $\mathcal O(\epsilon_\lambda)$-neighborhoods of its initial values for sufficiently small $\epsilon_\lambda$; further considering eq.~(\ref{eq:lambdastarforunshiftedlambda0}), 
there is an $\mathcal O(1)$ positive constant $c_\lambda$ such that 
\begin{equation}
    \frac{1}{\lambda_*(\lambda_{0{\rm init}}-\epsilon_\lambda)}\in (0,c_\lambda\epsilon_\lambda)+\left(\lambda_{0{\rm init}}^{-1}-c_\lambda\epsilon_\lambda,\lambda_{0{\rm init}}^{-1}+c_\lambda\epsilon_\lambda\right). 
\end{equation}
The zero point of $\partial_\lambda\langle\ln P({\bf d};\tilde \beta_{\rm d},\lambda)\rangle_{\boldsymbol\beta_0}$ within the $\mathcal O(\epsilon_\lambda)$-neighborhood of $\lambda_0$ is uniquely $\lambda_*=\lambda_0$, given the local concavity of $\langle \ln P({\bf d};\boldsymbol\beta)\rangle_{\boldsymbol\beta_0}$ [negative definite $\partial_{\boldsymbol\beta}^2\langle \ln P({\bf d};\boldsymbol\beta)\rangle_{\boldsymbol\beta_0}=\langle \partial_{\boldsymbol\beta}^2\ln P({\bf d};\boldsymbol\beta)\rangle_{\boldsymbol\beta_0}$] at $\boldsymbol\beta=\boldsymbol\beta_0$ (eq.~\ref{eq:boundforFisherinfo}). 
Besides, $\lambda_*=\infty$ is a stable fixed point of the $\lambda_*$ flow driven by eq.~(\ref{eq:floweqforlambdastar}), where the right hand side is zero for $\lambda_*=\infty$ and positive for $\lambda_*<\infty$, and hence the zero point at infinity $\lambda_*=\infty$ does not generate proximate $\lambda_*$ in the real space. 
Consequently, we derive 
\begin{equation}
    \lambda_*(\lambda_{0{\rm init}}-\epsilon_\lambda)= \lambda_{0{\rm init}}-\epsilon_\lambda,\infty. 
    \label{eq:lambdastarforshiftedlambda0}
\end{equation}
Recursive derivation of eq.~(\ref{eq:lambdastarforshiftedlambda0}) from eq.~(\ref{eq:lambdastarforunshiftedlambda0}) inducts $\lambda_*(\lambda_0)=\lambda_0,\infty$ below near infinite $\lambda_0=\lambda_{0{\rm init}}$ where eq.~(\ref{eq:lambdastarforunshiftedlambda0}) follows eqs.~(\ref{eq:expandedgradientformoderatelambda}) and (\ref{eq:expandedgradientforlargelambda}) with eq.~(\ref{eq:differentiatedprofilelikelihoodlikefunc}), and 
especially for the positive real space of $\lambda$ ($0<\lambda<\infty$) we consider, 
\begin{equation}
    \lambda_*(\lambda_0)=\lambda_0.
    \label{eq:uniquestationarity_lambda}
\end{equation}
Equations~(\ref{eq:implicativebetadform}) and (\ref{eq:uniquestationarity_lambda}) show the stationary point of the cross entropy $\langle-\ln P({\bf d};\boldsymbol\beta)\rangle_{\boldsymbol\beta_0}$
(${\bf 0}=\partial_{\boldsymbol\beta}\langle\ln P({\bf d};\boldsymbol\beta)\rangle_{\boldsymbol\beta_0}$) 
is uniquely the mode of $\langle\ln P({\bf d};\boldsymbol\beta)\rangle_{\boldsymbol\beta_0}$, that is $\boldsymbol\beta=\boldsymbol\beta_0$: 
\begin{equation}
    \left.\frac{\partial \langle-\ln P({\bf d};\boldsymbol\beta)\rangle_{\boldsymbol\beta_0}}{\partial \boldsymbol\beta}\right|_{\boldsymbol\beta\neq\boldsymbol\beta_0}\neq {\bf 0}.
    \label{eq:uniquestationarity}
\end{equation}
Equation~(\ref{eq:uniquestationarity}) also indicates the stationary point of the Kullback-Leibler divergence of $P({\bf d};\boldsymbol\beta)$ from $P({\bf d};\boldsymbol\beta_0)$ is here uniquely $\boldsymbol\beta=\boldsymbol\beta_0$:
$
    \left.
    \partial_{\boldsymbol\beta} D[P({\bf d};\boldsymbol\beta)||P({\bf d};\boldsymbol\beta_0)]\right|_{\boldsymbol\beta\neq\boldsymbol\beta_0}\neq {\bf 0}
$. 

The log marginal likelihood of the hyperparameters $\ln P({\bf d};\boldsymbol\beta)$ has shown in its leading order identical to the average (eq.~\ref{eq:deterministiccrossentropy}), 
where we further saw the mode $\boldsymbol\beta=\boldsymbol\beta_0$ is a unique stationary point (eq.~\ref{eq:uniquestationarity}) and has a sufficiently large curvature (eq.~\ref{eq:boundforFisherinfo}). 
These anticipate and actually derive the desired consistency of the MMLE. 
The body part of the proof begins with showing the modality of $\boldsymbol\beta=\boldsymbol\beta_0$ in $\ln P({\bf d};\boldsymbol\beta_0)$ in a strong sense; considering a sufficiently small positive constant $\epsilon$, 
for $|\boldsymbol\beta-\boldsymbol\beta_0|>\epsilon$, we have
\begin{equation}
  \langle\ln P({\bf d};\boldsymbol\beta_0)\rangle_{\boldsymbol\beta_0}
  -
  \langle\ln P({\bf d};\boldsymbol\beta)\rangle_{\boldsymbol\beta_0}
  \geq  \frac{(N+P-M)\delta_{H}}{2}(\epsilon/\beta_{\rm d0})^2, 
  \label{eq:strongmodality}
\end{equation} 
where $\delta_{H}$ is a $\Theta (1)$ positive number such that 
\begin{equation}
    \delta_{H}:=\max \left[\left[6\left(1+\frac{2N^\prime-{\rm Tr}({\bf D}_0)}{N^\prime-{\rm Tr}({\bf D}_0)}\lambda_0\right)\right]^{-2},
    \frac{c_{\rm det}(\lambda_0)}{27}\left[\lambda_0+\frac{N^\prime-{\rm Tr}({\bf D}_0)}{N^\prime}(1+\lambda_0)\right]^{-2}
    \right].
    \label{eq:defofdeltaH}
\end{equation}
Note $N^\prime -{\rm Tr}({\bf D}_0)={\rm Tr}({\bf H CH}^{\rm T}{\bf E}^{-1}{\bf D}_0{\bf D}_0^\dag)>0$. 
Inequality~(\ref{eq:strongmodality}) naturally follows the uniqueness of the stationary point in $\langle -\ln P({\bf d};\boldsymbol\beta)\rangle_{\boldsymbol\beta_0}$ being an $\mathcal O(N+P-M)$ variable, and the derivation of eq.~(\ref{eq:strongmodality}) appears in the supplement~\citep{sato2022supplement} to obtain an explicit expression of $\delta_{H}$ (eq.~\ref{eq:defofdeltaH}). 
Equation~(\ref{eq:strongmodality}) concludes the local maximum $\boldsymbol\beta=\boldsymbol\beta_*$ within the $\epsilon$-neighborhood of $\boldsymbol\beta_0$ guaranteed to exist by Theorem~\ref{th:localsufficientconvexity} is the mode in $P({\bf d};\boldsymbol\beta)$, given the almost deterministic nature of the empirical cross entropy within the model family of $P({\bf d};\boldsymbol\beta)$ (eq.~\ref{eq:deterministiccrossentropy}).

\begin{theorem}[Consistency of the MMLE]\label{th:consistencyingenlinmodels}
The MMLE of the hyperparameters $\boldsymbol\beta$ is consistent in the two-hyperparameter general linear models (eqs.~\ref{eq:twohyperparameterform})
with large degrees of freedom $M\propto N$,
as long as the associated regularized least square estimate is well-posed. 
\end{theorem}

\begin{proof}[Proof of Theorem~\ref{th:consistencyingenlinmodels}]
Defining $\delta\ln P({\bf d};\boldsymbol\beta)=\ln P({\bf d};\boldsymbol\beta)-\langle\ln P({\bf d};\boldsymbol\beta)\rangle_{\boldsymbol\beta_0}$ with 
the true data average $\langle\cdot\rangle:=\int d{\bf d}\cdot Q({\bf d})$ by $Q({\bf d})$,
for $|\boldsymbol\beta-\boldsymbol\beta_0|>\epsilon$, we have
\begin{equation}
\begin{aligned}
    &
    {\rm Prob}\left[\ln P({\bf d};\boldsymbol\beta_0)
    -
    \ln P({\bf d};\boldsymbol\beta)\leq 0\right]
    \\
    \leq&
    {\rm Prob}\left[(N+P-M)\delta_{H}\epsilon^2/2+\delta\ln P({\bf d};\boldsymbol\beta_0)-
    \delta\ln P({\bf d};\boldsymbol\beta) \leq 0\right]
    \\
    \leq&
    {\rm Prob}\left[(N+P-M)\delta_{H}\epsilon^2/2-2\max[|\delta\ln P({\bf d};\boldsymbol\beta)|,
    |\delta\ln P({\bf d};\boldsymbol\beta_0)|] \leq 0\right]
    \\\to& 0 \hspace{4pt}(N,M,P\to\infty),
\end{aligned}
\label{eq:maximalityweakform}
\end{equation}
where the first and last transform utilize eqs.~(\ref{eq:strongmodality}) and (\ref{eq:deterministiccrossentropy}), respectively. 
Our proof is almost completed by eq.~(\ref{eq:maximalityweakform}) that supposes a single value set of $\boldsymbol\beta$; 
because the $\epsilon$-proximity between $\boldsymbol\beta_0$ and $\hat{\boldsymbol\beta}$ is realized when a local maximum exists within the $\epsilon$-neighborhood of $\boldsymbol\beta_0$ and is larger than the probability values at the other stationary points $\boldsymbol\beta_{*n}$ if $\boldsymbol\beta_{*n}$ exist (although such a probability is negligible as mentioned below), we have the following evaluation using $\epsilon_{\beta}<\min[\epsilon/\sqrt{2},\beta_{d0}, \beta_{a0}]$:
\begin{equation}
\begin{aligned}
    {\rm Prob}[|\hat{\boldsymbol\beta}-\boldsymbol\beta_0|\leq\epsilon]\geq {\rm Prob}\left[
    \bigcap_{\hat x=\hat x_{\rm d},\hat x_{\rm a}} \bigcap_{\pm=\pm1}
    \left(\pm\hat x^{\rm T}\partial_{\boldsymbol\beta}\ln P ({\bf d};\boldsymbol\beta_0\pm \epsilon_\beta\hat x)\leq 0\right)
    \right.\\\left.
    \cap
    \max_{n;|\boldsymbol\beta_{*n}-\boldsymbol\beta_0|>\epsilon}\left(\ln P({\bf d};\boldsymbol\beta_0)
    -
    \ln P({\bf d};\boldsymbol\beta_{* n})>0\right)
    \right] 
    \\
    \geq 
    1-
    \sum_{\hat x=\hat x_{\rm d},\hat x_{\rm a}} \sum_{\pm=\pm1}{\rm Prob}\left[\pm\hat x^{\rm T}\partial_{\boldsymbol\beta}\ln P ({\bf d};\boldsymbol\beta_0\pm \epsilon_\beta\hat x)> 0
    \right]\\
    -\max_{n;|\boldsymbol\beta_{*n}-\boldsymbol\beta_0|>\epsilon}{\rm Prob}\left[\ln P({\bf d};\boldsymbol\beta_0)
    -
    \ln P({\bf d};\boldsymbol\beta_{*n})\leq 0
    \right] 
    \end{aligned}
    \label{eq:consistencyproofbody}
\end{equation}
where the subscript $n;|\boldsymbol\beta_{*n}-\boldsymbol\beta_0|>\epsilon$ represents $n$ such that $|\boldsymbol\beta_{*n}-\boldsymbol\beta_0|>\epsilon$; 
in the initial transform, the first proposition represents the existence of the local maximum (maxima) within the $\epsilon$-neighborhood, and the 
second proposition requires this local maximum (these maxima), equals to or larger than $P({\bf d};\boldsymbol\beta_0)$, to be larger than the other possible stationary points. 
In the asymptotics of ineq.~(\ref{eq:consistencyproofbody}) with eqs.~(\ref{eq:stationarityproofbody}) (Theorem~\ref{th:localsufficientconvexity}) and (\ref{eq:maximalityweakform}) for small $\epsilon$,
\begin{equation}
    \lim_{N,M,P\to\infty}{\rm Prob}[|\hat{\boldsymbol\beta}-\boldsymbol\beta_0|>\epsilon]=0,
\end{equation}
which was to be demonstrated. 

\end{proof}

The uniqueness of the stationary point in $\ln P({\bf d};\boldsymbol\beta)$ also follows for this model below a finite cutoff of $\lambda$  (denoted by $\lambda_c(<\infty)$), that is within a reasonable finite interval of $\boldsymbol\beta$, $\lambda<\lambda_c$. 
The following inequality may be useful; 
for any random variable $f$ and $\epsilon<\min\langle |f|\rangle$,
\begin{equation}
    {\rm Prob}[|f|<\epsilon]\leq
    {\rm Prob}[|f-\langle f \rangle |>\min\langle |f|\rangle-\epsilon].
    \label{eq:absoluteboundbyfluctuation}
\end{equation}
Inequality~(\ref{eq:absoluteboundbyfluctuation}) explains infinitesimal $f$ requires a stochastic deviation $f-\langle f\rangle$ by the absolute mean $\langle |f|\rangle$, which is bounded by $\min\langle |f|\rangle$. For $|\boldsymbol\beta-\boldsymbol\beta_0|>\epsilon\cap \lambda<\lambda_c$,  eqs.~(\ref{eq:deterministiccrossentropygradient}) and (\ref{eq:absoluteboundbyfluctuation}) yield
\begin{equation}
    \begin{aligned}
    &{\rm Prob}\left[\frac 1 {N^\prime} |\partial_{\boldsymbol\beta}\ln P({\bf d};\boldsymbol\beta)|<\epsilon\right]
    \\\leq&
    {\rm Prob}\left[\frac 1 {N^\prime} \left|\partial_{\boldsymbol\beta}\ln P({\bf d};\boldsymbol\beta)-\langle\partial_{\boldsymbol\beta}\ln P({\bf d};\boldsymbol\beta)\rangle\right|>\frac 1 {N^\prime} \min\left|\langle\partial_{\boldsymbol\beta}\ln P({\bf d};\boldsymbol\beta)\rangle\right|-\epsilon\right]
    \\\to& 0 \hspace{4pt} (N,M,P\to\infty).
    \end{aligned}
\end{equation}
In the last transform, we use $\partial_{\beta_{\rm d}}\ln P({\bf d};\boldsymbol\beta)=\mathcal O(N^\prime)$ for $|\lambda-\lambda_0|>\epsilon$ and
$\partial_\lambda\ln P({\bf d};\tilde\beta_{\rm d},\lambda)=\mathcal O(N^\prime)$ for $|\lambda-\lambda_0|>\epsilon/\beta_{\rm d0}\cap \lambda<\lambda_{\rm c}$; this conditional branch is parallel to the derivation of eq.~(\ref{eq:strongmodality}) thus not detailed here. The cross entropy $\langle-\ln P({\bf d};\boldsymbol\beta) \rangle$ is almost stationary at $\lambda\to\infty$, $\beta_{\rm d}\to\tilde\beta_{\rm d}(\infty)$, and then we require the additional assumption of finite $\lambda$ here for showing the above bound. 
The same logic applies to multiple $\boldsymbol\beta$ values even with arbitrarily large $\lambda_{\rm c}$, and the uniqueness of the stationary point, which at least exists with probability 1 given Theorem~\ref{th:localsufficientconvexity}, formally follows in the unlimited original $\boldsymbol\beta$ ($\lambda$) domain. 

\begin{theorem}[Unique existence of the stationary point in the empirical and generalization cross-entropy losses]\label{th:uniquestationarityingenlinearmodels}
The two-hyperparameter general linear models (eqs.~\ref{eq:twohyperparameterform}) defines 
a unique stationary point in 
$\langle \ln P({\bf d};\boldsymbol\beta)\rangle_{\boldsymbol\beta_0}$ 
and asymptotically almost surely  also in
$\ln P({\bf d};\boldsymbol\beta)$ for ${\bf d}\sim P({\bf d};\boldsymbol\beta_0)$, 
as long as the associated regularized least square estimate is well-posed. 
\end{theorem}

Besides the two-hyperparameter case (eqs.~\ref{eq:twohyperparameterform}) of the general linear model analyzed in this section, 
its one-hyperparameter form is also often employed for the regularized least square where $\beta_{\rm d}$ or $\beta_{\rm a}$ ($\sigma^2$ or $\rho^2$) or $\lambda$ is a unique unknown. 
For known $\beta_{\rm d}$, the stationary point search utilizes the following instead of eq.~(\ref{eq:differentiatedprofilelikelihoodlikefunc}) for the extremum condition with respect to $\lambda$:
\begin{equation}
    \langle\partial_\lambda\ln P({\bf d};\beta_{{\rm d}0},\lambda)\rangle_{\boldsymbol\beta_0}=\frac{1}{2\lambda}{\rm Tr}[{\bf HC}^\prime{\bf H}^{\rm T}{\bf E}^{-1}({\bf I}-{\bf D}{\bf D}_0^\dag)],
    \label{eq:differentiatedprofilelikelihoodlikefunc_knownbetad}
\end{equation}
which also defines the unique stationary point in $\langle\ln P({\bf d};\boldsymbol\beta)\rangle$; 
likewise, eq.~(\ref{eq:implicativebetadform}) uniquely determines the stationary point as $\beta_{\rm d0}$ for the case of
known $\beta_{\rm a}$ (or known $\lambda$). 
Thus, the same analysis proves the consistency of $\hat \beta$ [besides the asymptotic almost sure uniqueness of the stationary point in $\ln P({\bf d};\boldsymbol\beta)$] in one-hyperparameter cases. 
A simple proof with known $\sigma^2$ is previously shown in Ref.~\citep{mukhopadhyay2003parametric}. 
Remarkably, the uniqueness of the zero point in eq.~(\ref{eq:differentiatedprofilelikelihoodlikefunc_knownbetad}) holds even with ${\bf E}\tilde {\bf E}={\bf I}$ ($N=P=M\cap{\bf G}\propto{\bf H}^{\rm T}{\bf E}^{-1}$, yielding an ill-posed regularized least square estimate ${\bf a}_*$), which makes $\lambda$ unconstrained when both $\beta_{\rm d}$ and $\beta_{\rm a}$ are unknown. It asserts the reason for the consistency loss of $\hat{\boldsymbol\beta}$ in the ill-posed two-hyperparameter analysis may be that $\beta_{\rm d}$ and $\beta_{\rm a}$ are indistinguishable from the data in that ill-posed model. 
We last emphasize that our analysis with the zero-mean prior is applicable to the prior of non-zero means, by redefining the model parameters ${\bf a}^\prime$ with the prior average ${\bf a}_{\rm prior}^\prime$ as ${\bf a}^\prime:={\bf a}+{\bf a}_{\rm prior}^\prime$. 

Since $\hat{\boldsymbol\beta}$ is a consistent estimator, 
the regularized least square ${\bf a}_*$ using $\hat\lambda$ is asymptotically unbiased. 
Considering the stability of ${\bf a}_*$ for $\lambda>0$ in the well-posed case with
$N+P-M>0$ and ${\bf G}\not\propto{\bf H}^{\rm T}{\bf E}^{-1}{\bf E}$, ${\bf a}_*$ asymptotically combines the unbiasedness and stability: 
\begin{equation}
\begin{aligned}
    \langle a_{0i}-a_{*i}(\hat\lambda)\rangle_{\boldsymbol\beta_0}&\to 0
    \\
    \frac 1 N\sum_{i}\langle |a_{0i}- a_{*i}(\hat\lambda)|^2\rangle_{\boldsymbol\beta_0}&<\infty
\end{aligned}
\end{equation}
It contrasts with the the maximum likelihood estimate ${\bf a}_*(0)$ being unbiased but unstable when $N\leq M$:
    $\langle a_{0i}-a_{*i}(0)\rangle_{\boldsymbol\beta_0}
    =0$, 
    $\frac 1 N\sum_{i}\langle |a_{0i}- a_{*i}(0)|^2\rangle_{\boldsymbol\beta_0}
    \to\infty$.

\section{Essentials and generalization}
In this section, we generalize the proof of consistency and highlight the framework of our proof detailed in the previous section for the general linear models. 
The essential postulates of the proof are found in the extensive properties of the cost functions (\S\ref{sec:gen1}) and that of the empirical cross entropy (\S\ref{sec:gen_efficiency}). 
The proof is immediately generalized for nonlinear multiple cost functions with extensive properties in estimating the hyperparameters $\boldsymbol\beta_{\rm d}$ of the model-parameter likelihood (\S\ref{sec:gen2}). 
The asymptotic normality of the predictive distribution further extends the applicability even outside the exponential families, besides even to the true data distribution outside the model space (\S\ref{sec:gen_efficiency_app}). 

\subsection{Almost sure almost stationarity of the marginal likelihood of the hyperparameters and extensive properties of Gibbsian cost functions}
\label{sec:gen1}
Adopting the same logic as in the Gaussian ones, 
we obtain a similar results for general Gibbsian models (eqs.~\ref{eq:modeldefinition}). 

Similarly to eq.~(\ref{eq:fluctuationprobabilitybound}), 
Chebyshev's inequality and the variance decomposition yield
\begin{equation}
    {\rm Prob}\left[\left|\frac{\langle U_i\rangle_{\boldsymbol\beta_0|{\bf d}}}{\langle U_i\rangle_{\boldsymbol\beta_0}}-1\right|>\epsilon\right]
    \leq 
    \frac{{\rm Var}[U_i;\boldsymbol\beta_0]}{\epsilon^2\langle U_i\rangle_{\boldsymbol\beta_0}^2}
    \label{eq:fluctuationprobabilitybound_genlikelihood}
\end{equation}
and
\begin{equation}
    {\rm Prob}\left[\left|\frac{\langle V_i\rangle_{\boldsymbol\beta_0|{\bf d}}}{\langle V_i\rangle_{\boldsymbol\beta_0}}-1\right|>\epsilon\right]
    \leq 
    \frac{{\rm Var}[V_i;\boldsymbol\beta_0]}{\epsilon^2\langle V_i\rangle_{\boldsymbol\beta_0}^2}
    \label{eq:fluctuationprobabilitybound_genprior}
\end{equation}
The following lemma then holds with applicability to various Gibbsian statistical models. 

\begin{lemma}\label{le:almostsurealmoststationarity}
When the first and second cumulants of the cost functions are extensive, 
\begin{equation}
\begin{aligned}
    \forall i, \langle U_i\rangle_{\boldsymbol\beta}, {\rm Var}[U_i;\boldsymbol\beta]=\Theta (N)
    \\
    \langle V_j\rangle_{\boldsymbol\beta},{\rm Var}[V_j;\boldsymbol\beta]=\Theta (M)
    \label{eq:assumptionofalmostsurealmoststationarity}
\end{aligned}
\end{equation}
the marginal likelihood of the hyperparameters is almost surely stationary in leading orders:
\begin{equation}
    \lim_{N\to\infty} {\rm Prob}\left[\left|\frac{\langle U_i\rangle_{\boldsymbol\beta_0|{\bf d}}}{\langle U_i\rangle_{\boldsymbol\beta_0}}-1\right|>\epsilon\right]=\lim_{M\to\infty} {\rm Prob}\left[\left|\frac{\langle V_j\rangle_{\boldsymbol\beta_0|{\bf d}}}{\langle V_j\rangle_{\boldsymbol\beta_0}}-1\right|>\epsilon\right]=0,
    \label{eq:generalalmostsurealmoststationarity_costfunctionform}
\end{equation}
and for fixed $M/N$,
\begin{equation}
    \lim_{N,M\to\infty} {\rm Prob}\left[\frac 1 N\left|\partial_{\boldsymbol\beta}\ln P({\bf d};\boldsymbol\beta)\right|_{\boldsymbol\beta=\boldsymbol\beta_0}>\epsilon\right]=0.
    \label{eq:generalalmostsurealmoststationarity}
\end{equation}
\end{lemma}

Lemma~\ref{le:almostsurealmoststationarity} immediately assures the consistency of $\hat{\boldsymbol\beta}$ for sufficiently concave $\ln P({\bf d};\boldsymbol\beta)$.

\begin{theorem}[Consistency of the MMLE in the sufficiently globally concave marginal likelihood]\label{th:genconsistencywithsufficientconvexity}
When $\ln P({\bf d};\boldsymbol\beta)$ is globally concave in its leading order with respect to $\boldsymbol\beta$, or more weakly,
\begin{equation}
    \max_{\boldsymbol\beta}\lim_{N\to\infty}\frac 1 N \partial^2_{\beta_i} \ln P({\bf d};\boldsymbol\beta)<0,
    \label{eq:assumptionofsufficientglobalconvexity}
\end{equation}
then the MMLE is a consistent estimator for the hyperparameters conjugated with the extensive cost functions that satisfy eq.~(\ref{eq:assumptionofalmostsurealmoststationarity}). 
\end{theorem}

\begin{proof}[Proof of Theorem~\ref{th:genconsistencywithsufficientconvexity}]
Theorem~\ref{th:genconsistencywithsufficientconvexity} is derived as follows:
\begin{equation}
\begin{aligned}
    {\rm Prob}[|\hat{\boldsymbol\beta}-\boldsymbol\beta_0|\leq\epsilon]\geq{\rm Prob}\left[
    \bigcap_i \bigcap_{\pm=\pm1}
    \left(\pm\partial_{\beta_i}\ln P ({\bf d};\boldsymbol\beta_0\pm \epsilon_\beta\hat x_{\beta_i})\leq 0\right)
    \right] 
    \\
    \geq 
    1-
    \sum_{i} \sum_{\pm=\pm1}{\rm Prob}\left[\pm\partial_{\beta_i}\ln P ({\bf d};\boldsymbol\beta_0\pm \epsilon_\beta\hat x_{\beta_i})> 0
    \right]
    \\
    \geq 
    1-
    \sum_{i} \sum_{\pm=\pm1}{\rm Prob}\left[\pm\frac 1 N \partial_{\beta_i}\ln P ({\bf d};\boldsymbol\beta_0)> -\frac 1 N \int_0^{\pm\epsilon_\beta} dl \partial^2_{\beta_i}\ln P ({\bf d}; \boldsymbol\beta_0+\hat x_{\beta_i}l)
    \right]
    \\\to 1 \hspace{4pt}(N,P,M\to\infty).
    \end{aligned}
    \label{eq:proof_consistencyforsufficientconvexity}
\end{equation}
The uniqueness of the stationary point due to the assumed global concavity of $\ln P({\bf d};\boldsymbol\beta)$ is used in the first transform, and eqs.~(\ref{eq:generalalmostsurealmoststationarity}) and (\ref{eq:assumptionofsufficientglobalconvexity}) are used in the last transform. 
\end{proof}

Although not applicable to the general linear model analyzed in the previous section, where we saw $\ln P({\bf d};\boldsymbol\beta)$ is not concave along $\beta=\tilde \beta_{\rm d}$, Theorem~\ref{th:genconsistencywithsufficientconvexity} generally covers the hyperparameters of the Gibbsian model-parameter likelihood, as shown later in \S\ref{sec:gen2}. 

The assumption of Lemma~\ref{le:almostsurealmoststationarity} is associated with the proportionality of the mean and variance of the cost functions $U_i$ and $V_i$ to the dimensions of the data and model parameter spaces. 
The macroscopic variables in the statistical thermodynamics~\citep{landau1994statistical} are phrased to be extensive 
when in their leading orders proportional to the 
sample-space dimension of the stochastically treated particle states; the following relation describes the leading-order proportionality of 
the extensive variables $Y(L)$ to the dimension $\dim (x)=L$ of a random variable $x$:
\begin{equation}
    \lim_{L\to\infty} \left[\frac{Y(L)}L-\frac{Y(cL)}{cL}\right]=0 \hspace{10pt}({\rm extensive\,\, property})
    \label{eq:extensivecond}
\end{equation}
Equation~(\ref{eq:extensivecond}) entails the existence of an asymptotic $N$-independent function $y$ such that $Y(L)=Ly+o(L)$: 
\begin{equation}
    \lim_{L\to\infty} \left[\frac{Y(L)}L-y\right]=0.
\end{equation}
The cumulants of the cost functions are often extensive in a wide range of the Gibbs ensembles. Examples include the Gaussian ones (eqs.~\ref{eq:twohyperparameterform}), where the distribution of $2U$ in $P({\bf d}|{\bf a};\boldsymbol\beta_{\rm d})$ follows the $\chi^2$-distribution with $N$ degrees of freedom, 
and
the effective cost function 
$U+\lambda V-\min_{{\bf a}|{\bf d}}[U+\lambda V]$, subtracting the minimum value of the cost function $U+\lambda V$, also follows the $\chi^2$-distribution with $M$ degrees of freedom in the model-parameter posterior $P({\bf a}|{\bf d};\boldsymbol\beta)$ after multiplied by two, similarly to $2V$ in the model-parameter prior $P({\bf a};\boldsymbol\beta_{\rm a})$; here the minimum value subtraction is associated with the arbitrariness of the ground-state value in the cost functions.

\subsection{Consistency with model-selection efficiency
under the effective equivalence between the empirical and generalization cross-entropy losses}
\label{sec:gen_efficiency}

The proof of consistency in the previous section is besides based on the extensive properties of the cross entropy and Fisher information of the marginal likelihood of the hyperparameters, which are actually true for the two-hyperparameter general linear models (eqs.~\ref{eq:twohyperparameterform}). 

The consistency of the MMLE may be obvious when 
the extensive property of the empirical cross entropy is satisfied in a deterministic manner; if
the empirical cross entropy of $P({\bf d};\boldsymbol\beta)$ relative to $P({\bf d};\boldsymbol\beta_0)$ for one data sample ${\bf d}$ satisfies the leading-order convergence in probability, 
\begin{equation}
    \lim_{N,M\to\infty}{\rm Prob}\left[\left|
    -\frac 1 N\ln P({\bf d};\boldsymbol\beta)-
    \left\langle -\frac 1 N\ln P({\bf d};\boldsymbol\beta)\right\rangle
    \right|>\epsilon
    \right]=0,
    \label{eq:extensiveempiricalcrossentropy}
\end{equation}
then the MMLE is consistent, that is 
\begin{equation}
    \lim_{N,M\to\infty}{\rm Prob}\left[\left|
    {\boldsymbol\beta}-\boldsymbol\beta_0
    \right|>\epsilon
    \right]=0,
\end{equation}
as far as the global minimum is well-defined in $\lim_{N\to\infty}N^{-1}\langle -\ln P({\bf d};\boldsymbol\beta)\rangle_{\boldsymbol\beta_0}$. 
Equation~(\ref{eq:extensiveempiricalcrossentropy}) means there is an asymptotic $N$-independent function $h(\boldsymbol\beta)$ such that
\begin{equation}
    h(\boldsymbol\beta)=\lim_{N,M\to\infty}\left\langle -\frac 1 N\ln P({\bf d};\boldsymbol\beta)\right\rangle_{\boldsymbol\beta_0}.
    \label{eq:largedeviationcrossentropy}
\end{equation}
Equation~(\ref{eq:extensiveempiricalcrossentropy}) with $h(\boldsymbol\beta)$ 
provides a bound of the $\ln P({\bf d};\boldsymbol\beta_0)$ difference; for $|\boldsymbol\beta-\boldsymbol\beta_0|>\epsilon$, 
\begin{equation}
\begin{aligned}
    \frac 1 N [\ln P({\bf d};\boldsymbol\beta_0)-\ln P({\bf d};\boldsymbol\beta)]
    >&\frac 1 N \min_{|\boldsymbol\beta-\boldsymbol\beta_0|>\epsilon}[\ln P({\bf d};\boldsymbol\beta_0)-\ln P({\bf d};\boldsymbol\beta)]
    \\
    \xrightarrow{p}& \min_{|\boldsymbol\beta-\boldsymbol\beta_0|>\epsilon}[h(\boldsymbol\beta_0)-h(\boldsymbol\beta)] \hspace{4pt}(N\to\infty),
    \end{aligned}
    \label{eq:schematicextensivityofcrossentropydiff}
\end{equation}
where the convergence in probability under $Q({\bf d})$ is represented by $p$ over an arrow. 
Equation~(\ref{eq:extensiveempiricalcrossentropy}) assuming the well-defined global minimum of $h(\boldsymbol\beta)$ also means the positivity of the observed Fisher information in its leading order; 
\begin{equation}
    \pm\frac 1 N \partial_{\beta_i}\ln P({\bf d};\boldsymbol\beta_0\pm\epsilon_\beta\hat x_{\beta_i})\xrightarrow{p}
    \epsilon_\beta \partial_{\beta_i}^2h(\boldsymbol\beta_0)+\mathcal O(\epsilon_\beta^2)<0 \hspace{4pt}(N,M\to\infty),
    \label{eq:schematicconvergenceofFisherinfo}
\end{equation}
where the last inequality is evaluated with sufficiently small $\epsilon_\beta$. 
Here we may take constant $M/N$, considering
$-\partial_{\beta_i}^2h(\boldsymbol\beta_0)$ in Gibbsian models represents the interclass variance of the cost function divided by $N$, which can be finite for both $U_i$ and $V_i$ when $M/N$ is constant. 
Equations~(\ref{eq:schematicextensivityofcrossentropydiff}) and (\ref{eq:schematicconvergenceofFisherinfo}) assure the consistency of the MMLE as in other cases shown earlier (cf. eqs.~\ref{eq:maximalityweakform}, \ref{eq:consistencyproofbody}, and \ref{eq:proof_consistencyforsufficientconvexity}); using $\epsilon_{\beta}<\min[\epsilon/\sqrt{{\rm dim}(\boldsymbol\beta)},\min_i(\beta_{0i})]$, 
\begin{equation}
\begin{aligned}
    {\rm Prob}[|\hat{\boldsymbol\beta}-\boldsymbol\beta_0|\leq\epsilon]\geq{\rm Prob}\left[
    \bigcap_i \bigcap_{\pm=\pm1}
    \left(\pm\partial_{\beta_i}\ln P ({\bf d};\boldsymbol\beta_0\pm \epsilon_\beta\hat x_{\beta_i})\leq 0\right)
    \right.\\\left.
    \cap
    \max_{n;|\boldsymbol\beta_{*n}-\boldsymbol\beta_0|>\epsilon}\left(\ln P({\bf d};\boldsymbol\beta_0)
    -
    \ln P({\bf d};\boldsymbol\beta_{* n})>0\right)
    \right] 
    \\
    \geq 
    1-
    \sum_{i} \sum_{\pm=\pm1}{\rm Prob}\left[\pm\frac 1 N \partial_{\beta_i}\ln P ({\bf d};\boldsymbol\beta_0)> -\frac 1 N \int_0^{\pm\epsilon_\beta} dl \partial^2_{\beta_i}\ln P ({\bf d}; \boldsymbol\beta_0+\hat x_{\beta_i}l)
    \right]
    \\
    -\max_{n;|\boldsymbol\beta_{*n}-\boldsymbol\beta_0|>\epsilon}
    {\rm Prob}\left[\frac 1 N \langle\ln P({\bf d};\boldsymbol\beta_0)
    -
    \ln P({\bf d};\boldsymbol\beta_{*n})\rangle
    \right.
    \\
    \left.
    -\frac 1 N [|\delta\ln P({\bf d};\boldsymbol\beta_0)|+|\delta\ln P({\bf d};\boldsymbol\beta_{*n})|] \leq 0\right]
    \\\to 1 \hspace{4pt}(N,M\to\infty).
    \end{aligned}
    \label{eq:consistencywithextensivecrossentropy}
\end{equation}
Since the mode of the cross entropy is the true hyperparameter set, 
the determinacy of the empirical cross entropy  (eq.~\ref{eq:extensiveempiricalcrossentropy}) almost implies the consistency of $\hat {\boldsymbol\beta}$. 
The analysis in \S\ref{sec:proofwithoutassuminguniquestationarity} validates eq.~(\ref{eq:extensiveempiricalcrossentropy}) with well-behaved $\lim_{N\to\infty}N^{-1}\langle -\ln P({\bf d};\boldsymbol\beta)\rangle_{\boldsymbol\beta_0}$ in a general linear model. 

The leading-order convergence in probability of the empirical cross entropy (eq.~\ref{eq:extensiveempiricalcrossentropy}) also 
leads to the Kullback-Leibler minimization nature of the MMLE even if the true data distribution $Q({\bf d})$ is outside the model family of $P({\bf d};\boldsymbol\beta)$. 
A model selection algorithm is occasionally phrased to be efficient when minimizes a given risk function~\citep{shibata1981optimal} asymptotically with probability 1, especially for the Kullback-Leibler divergence risk of $P({\bf d};\boldsymbol\beta)$ from $Q({\bf d})$~\citep{watanabe2018mathematical}:
\begin{equation}
    \lim_{N\to\infty} {\rm Prob}[|\hat{\boldsymbol\beta}-\boldsymbol\beta_{\rm 0}^\prime|>\epsilon]=0\hspace{20pt}({\rm model\mbox{-}selection \,\,efficiency}),
    \label{eq:defofmodelselectionefficiency}
\end{equation}
where
\begin{equation}
\boldsymbol\beta_0^\prime:={\rm argmin}_{\boldsymbol\beta}\mathcal D[P({\bf d}|\boldsymbol\beta)||Q({\bf d})].
\label{eq:defofeffectivetruehyperparameter}
\end{equation}
We may specifically call the condition of eq.~(\ref{eq:defofmodelselectionefficiency}) the model-selection efficiency, 
to distinguish the ordinary use of the term {\it efficiency} that refers to the consistent estimator of the minimum variance~\citep{gelman2013bayesian}. 
The concept of the model-selection efficiency is a generalization of the consistency extended to the true data distribution not included in the fitting model family [$Q({\bf d})\not\in \{P({\bf d}|\boldsymbol\beta)\}_{\boldsymbol\beta}$], where
$\boldsymbol\beta_0^\prime$ (eq.~\ref{eq:defofeffectivetruehyperparameter}) is contrasted with the true hyperparameters $\boldsymbol\beta_0$ such that $D[P({\bf d}|\boldsymbol\beta_0)||Q({\bf d})]=0$. 
The model-selection efficiency means the estimate $\hat {\boldsymbol\beta}$ provides the prior predictive distribution closest to the true form and is preferred for data predictability. 
The associated lemma is analogous to that of eq.~(\ref{eq:consistencywithextensivecrossentropy});
when
\begin{equation}
    \lim_{N,M\to\infty}{\rm Prob}\left[\left|
    -\frac 1 N\ln P({\bf d};\boldsymbol\beta)-
    \left\langle -\frac 1 N\ln P({\bf d};\boldsymbol\beta)\right\rangle
    \right|>\epsilon
    \right]=0,
\end{equation}
and as far as $h^\prime(\boldsymbol\beta)=\lim_{N,M\to\infty}\left\langle -\frac 1 N\ln P({\bf d};\boldsymbol\beta)\right\rangle$ has the single global minimum [$h^\prime(\boldsymbol\beta)<h^\prime(\boldsymbol\beta_0^\prime)$ for $\boldsymbol\beta\neq\boldsymbol\beta_0^\prime$], 
the MMLE is a model-selection efficient estimator of the hyperparameter set: 
\begin{equation}
    \lim_{N,M\to\infty}{\rm Prob}\left[\left|
    {\boldsymbol\beta}-\boldsymbol\beta_0^\prime
    \right|>\epsilon
    \right]=0.
\end{equation}
We obtain the above in the same manner as eq.~(\ref{eq:consistencywithextensivecrossentropy}) after replacing $\boldsymbol\beta_0$ with 
$\boldsymbol\beta_0^\prime$ and $P({\bf d};\boldsymbol\beta_0)$ with $Q({\bf d})$. 
In summary, when the empirical and generalization cross-entropy losses are effectively equivalent, as in the two-(or one-)hyperparameter general linear models (eqs.~\ref{eq:twohyperparameterform}), 
the MMLE reconciles the consistency with the efficiency in the model selection.

\begin{lemma}\label{le:genconsistencyandefficiencywithequivalenceincrossentropies}
When the empirical cross entropy of $P({\bf d};\boldsymbol\beta)$ relative to $Q({\bf d})$ satisfies the leading-order convergence in probability (eq.~\ref{eq:extensiveempiricalcrossentropy}), 
the MMLE is a consistent and model-selection efficient estimator of the hyperparameter set $\boldsymbol\beta$,
as far as the global minimum is well-defined in the leading order of the generalization cross-entropy loss $\lim_{N\to\infty}N^{-1}\langle -\ln P({\bf d};\boldsymbol\beta)\rangle$. 
\end{lemma}

The joint posterior of the model parameters and hyperparameters
$P({\bf a},{\boldsymbol\beta}|{\bf d})$ with a uniform hyperprior 
formally coincides with the product of $P({\bf a}|{\bf d};{\boldsymbol\beta})$ and $P({\bf d};\boldsymbol\beta)$ we have dealt with. Therefore, the consistency of $\hat{\boldsymbol\beta}$ means the following in such fully Bayesian analysis:
\begin{equation}
    P({\bf a},\hat{\boldsymbol\beta}|{\bf d})=P({\bf a}|{\bf d};{\boldsymbol\beta})P({\bf d};\boldsymbol\beta)\to
    P({\bf a}|{\bf d};\hat{\boldsymbol\beta})
    \delta({\boldsymbol\beta}-\hat{\boldsymbol\beta}) \,\,(N,M\to\infty)
    \label{eq:empiricalBaeys}
\end{equation}
The above asymptotic form eq.~(\ref{eq:empiricalBaeys}) is a substitution operation regarded as an approximation in the (parametric) empirical Bayes method~\citep{maritz2018empirical}. The consistency and model-selection efficiency of the MMLE $\hat{\boldsymbol\beta}$ in a large-degree-of-freedom model also mean that eq.~(\ref{eq:empiricalBaeys}) is an asymptotic analysis for a large number of data and model parameters.

\subsection{General consistency of the MMLE in estimating model-parameter-likelihood hyperparameters conjugated with extensive cost functions}
\label{sec:gen2}

Inference problems involve the two sample spaces of the model parameters and data, where the extensive variables are proportional to either of the two large parameters $N=\dim({\bf d})$ and $M=\dim({\bf a})$, as seen in the analysis of the general linear models. 
The cost function $V$ of the prior $P({\bf a};\boldsymbol\beta_{\rm a})$ is $\mathcal O(M)$ both in the prior $P({\bf a};\boldsymbol\beta_{\rm a})$ and posterior $P({\bf a}|{\bf d};\boldsymbol\beta)$, 
while the effective degree of the freedom in the cost function $U$ of the model-parameter likelihood $P({\bf d}|{\bf a};\boldsymbol\beta_{\rm d})$ (or more symmetrically, $\min_{\bf a}U$ and $\min_{\bf d}U$) is $N$ for $P({\bf d}|{\bf a};\boldsymbol\beta_{\rm d})$ and $M$ for $P({\bf a}|{\bf d};\boldsymbol\beta)$. 
This dimensional change of the sample space allows us to develop another proof of consistency in estimating the hyperparameters $\boldsymbol\beta_{\rm d}$ of the model-parameter likelihood. 
The following lemma assures the concavity of the marginal likelihood of the hyperparameters within the subspace spanned by the hyperparameters of the model-parameter likelihood. 
\begin{lemma}\label{le:convexityofthemarginalposteriorwrtmodelparameterlikelihoodhyperparameters}
For any cost function sets ${\bf U}$ of the model-parameter likelihood such that 
\begin{equation}
\begin{aligned}
    {\rm Var}[U_i;\boldsymbol\beta]&=\Theta(N)
    \\
    {\rm Var}[U_i|{\bf d};\boldsymbol\beta]&=\Theta(M)
\end{aligned}
\end{equation}
there is a critical number ratio $c_{*M/N}$ of the model parameters to the data below which 
all the diagonal components of 
$\partial_{\boldsymbol\beta}^2\ln P({\bf d};\boldsymbol\beta)$ are negative: 
\begin{equation}
    \partial_{\beta_{{\rm d}i}}^2\ln P({\bf d};\boldsymbol\beta)
    ={\rm Var}[U_i|{\bf d};\boldsymbol\beta]-{\rm Var}[U_i;\boldsymbol\beta]<0 \hspace{4pt}(M/N<c_{*M/N}).
\end{equation}
Here $c_{*M/N}$ is given by the following form:
\begin{equation}
    c_{*M/N}=\min_{i,\boldsymbol\beta}\frac{N^{-1}{\rm Var}[U_i;\boldsymbol\beta]}{M^{-1}{\rm Var}[U_i|{\bf d};\boldsymbol\beta]}.
    \label{eq:numberratiocutoff}
\end{equation}
\end{lemma}

The unbiased estimate of population variance provides an example of $c_{*M/N}=1$. 
Although $U_i$ may be proportional to the subspace dimension $N_i$, $U_i$ is treated as $\Theta (N)$ in Lemma~\ref{le:convexityofthemarginalposteriorwrtmodelparameterlikelihoodhyperparameters} for brevity. 
As Lemma~\ref{le:convexityofthemarginalposteriorwrtmodelparameterlikelihoodhyperparameters} leads to eq.~(\ref{eq:assumptionofsufficientglobalconvexity}), 
when additionally $\langle U_i\rangle=\Theta (N)$ is met, 
Lemma~\ref{le:convexityofthemarginalposteriorwrtmodelparameterlikelihoodhyperparameters} is followed by
the consistency of $\hat{\boldsymbol\beta}_{\rm d}$ 
through Theorem~\ref{th:genconsistencywithsufficientconvexity} if $\boldsymbol\beta_{\rm a0}$ is known. 

\begin{theorem}[Consistency of the MMLE for the model-parameter-likelihood hyperparameters]\label{th:genconsistencyformodelparameterlikelihoodhyperparameters}
For $M/N<c_{*M/N}$ with known $\boldsymbol\beta_{\rm a0}$, 
the MMLE is a consistent estimator for the hyperparameters $\boldsymbol\beta_{\rm d}$ conjugated with the cost functions ${\bf U}$ with extensive model means and variances $\langle U_i\rangle=\Theta (N)$, ${\rm Var}[U_i;\boldsymbol\beta]=\Theta(N)$ and posterior variances ${\rm Var}[U_i|{\bf d};\boldsymbol\beta]=\Theta(M)$: 
\begin{equation}
    \lim_{N,M\to\infty}{\rm Prob}\left[\left|
    \hat{\boldsymbol\beta}_{\rm d}-\boldsymbol\beta_{\rm d0}
    \right|>\epsilon
    \right]=0.
\end{equation}
\end{theorem}

The derivation of Theorem~\ref{th:genconsistencyformodelparameterlikelihoodhyperparameters} can be a proof of consistency for the hyperparameter MMLE in Bayesian models with fixed model-parameter-prior hyperparameters, as well as for the marginal likelihood method in frequentist approaches without the model-parameter prior, below the critical number ratio $c_{*M/N}$.  
Theorem~\ref{th:genconsistencyformodelparameterlikelihoodhyperparameters} may be regarded as an extension of the proof of the MMLE for ``parametric'' limits: $M/N\to0$. 
When $M/N\to0$, the extremal form (eq.~\ref{eq:Ibarel}) of the selection condition of the MMLE is simplified to $\min_{{\bf a}|{\bf d}}U_i/N=\langle U_i\rangle_{\boldsymbol\beta}/N$, as long as the first and second cumulant of the cost functions are extensive near $\boldsymbol\beta=\boldsymbol\beta_0$. 
It makes the MMLE and more simply $\min_{{\bf a}|{\bf d}}U_i=\langle U_i\rangle_{\boldsymbol\beta}$ the consistent estimators, because 
$\langle U_i\rangle_{\boldsymbol\beta}$ is a monotonically decreasing function of $\beta_{{\rm d}i}$ and thus $\min_{{\bf a}|{\bf d}}U_i=\langle U_i\rangle_{\boldsymbol\beta}$ has a unique solution. 
The change in the effective degree of freedom from $N$ in the likelihood to $M$ in the posterior seems to provide an intrinsic difference of the likelihood cost function $U_i$ from the prior cost function $V_i$ with $M$ degrees of freedom invariant between the prior and posterior. 

\subsection{Consistency and model-selection efficiency of the MMLE under the asymptotic normality of the predictive distribution applicable to non-exponential families}
\label{sec:gen_efficiency_app}
Lemma~\ref{le:genconsistencyandefficiencywithequivalenceincrossentropies} of the consistency and model-selection efficiency in \S\ref{sec:gen_efficiency} do not postulate the specific properties of the Gibbs distributions, thus principally applicable to the other model families. 
Here we see the MMLE is widely consistent and  model-selection efficient under the asymptotic normality~\citep{fujikoshi2011multivariate} of the data generator. 
The analysis in this section is not limited to the exponential family, but the same notation continues for brevity.

The leading-order equivalence between the empirical and generalization cross-entropy losses (eq.~\ref{eq:extensiveempiricalcrossentropy}) can follow the asymptotic normality of the prior predictive distribution $P({\bf d};\boldsymbol\beta)$ with respect to ${\bf d}$, which yields
\begin{equation}
\begin{aligned}
-\ln P({\bf d};\boldsymbol\beta)  \to
\langle -\ln P({\bf d};\boldsymbol\beta) \rangle_{\boldsymbol\beta}
+\frac 1 {2} 
({\bf d}-\langle {\bf d}\rangle_{\boldsymbol\beta})^{\rm T}{\bf C}_{\boldsymbol\beta}^{-1} ({\bf d}-\langle {\bf d}\rangle_{\boldsymbol\beta}) -\frac N 2 +o(N)
\\
({\rm asymptotic\,normality\,of\,the\,predictive\,distribution}),
\end{aligned}
\label{eq:asymptoticnormality}
\end{equation}
where ${\bf C}_{\boldsymbol\beta}$ denotes the covariance matrix of $P({\bf d};\boldsymbol\beta)$ such that
$(C_{\boldsymbol\beta})_{ij}={\rm Cov}[d_i:d_j;\boldsymbol\beta]$, here positive definite. 
When the true hyperparameter set $\boldsymbol\beta_0$ exists,
eq.~(\ref{eq:asymptoticnormality}) yields
\begin{equation}
\begin{aligned}
    -\frac 1 N \ln P({\bf d};\boldsymbol\beta)  \xrightarrow{p} 
    &
    \frac 1 N 
    \langle -\ln P({\bf d};\boldsymbol\beta) \rangle_{\boldsymbol\beta}
    \\
    &
    +\frac 1 {2N} 
    {\rm Tr}\{{\bf C}_{\boldsymbol\beta}^{-1}[{\bf C}_{\boldsymbol\beta_0}+(\langle {\bf d}\rangle_{\boldsymbol\beta}-\langle {\bf d}\rangle_{\boldsymbol\beta_0})(\langle {\bf d}\rangle_{\boldsymbol\beta}-\langle {\bf d}\rangle_{\boldsymbol\beta_0})^{\rm T}]\} -\frac 1 2 
\end{aligned}
\label{eq:extensivecrossentropyforasymptoticnormality_conditional}
\end{equation}
for $N\to\infty$ 
from Chebyshev's inequality with the following cumulant evaluation of $E_f:=({\bf d}-\langle {\bf d}\rangle_{\boldsymbol\beta})^{\rm T}{\bf C}_{\boldsymbol\beta}^{-1} ({\bf d}-\langle {\bf d}\rangle_{\boldsymbol\beta})/2$:
\begin{equation}
\begin{aligned}
    &\langle E_f
 \rangle_{\boldsymbol\beta_0}
=
\frac 1 2 {\rm Tr}\{{\bf C}_{\boldsymbol\beta}^{-1}[{\bf C}_{\boldsymbol\beta_0}+(\langle {\bf d}\rangle_{\boldsymbol\beta}-\langle {\bf d}\rangle_{\boldsymbol\beta_0})(\langle {\bf d}\rangle_{\boldsymbol\beta}-\langle {\bf d}\rangle_{\boldsymbol\beta_0})^{\rm T}]\}
+o(N)
\\
&
{\rm Var}[E_f;\boldsymbol\beta_0]
=
\frac 1 2
{\rm Tr}\{[{\bf C}_{\boldsymbol\beta}^{-1}({\bf C}_{\boldsymbol\beta_0}+(\langle {\bf d}\rangle_{\boldsymbol\beta}-\langle {\bf d}\rangle_{\boldsymbol\beta_0})(\langle {\bf d}\rangle_{\boldsymbol\beta}-\langle {\bf d}\rangle_{\boldsymbol\beta_0})^{\rm T})]^2\}
+o(N)
\end{aligned}
\end{equation}
which are both $\Theta (N)$. 
Equation~(\ref{eq:extensivecrossentropyforasymptoticnormality_conditional}) shows the leading-order equivalence between the empirical and  generalization cross-entropy losses premised to realize the consistency of the MMLE in Lemma~\ref{le:genconsistencyandefficiencywithequivalenceincrossentropies} (\S\ref{sec:gen_efficiency}). 
Even not assuming the existence of $\boldsymbol\beta_0$, the asymptotic normality of $P({\bf d};\boldsymbol\beta)$ and $Q({\bf d})$ realizes eq.~(\ref{eq:extensiveempiricalcrossentropy}):
\begin{equation}
\begin{aligned}
    -\frac 1 N \ln P({\bf d};\boldsymbol\beta)  \xrightarrow{p}& 
\frac 1 N 
\langle -\ln P({\bf d};\boldsymbol\beta) \rangle_{\boldsymbol\beta}
+
\frac 1 {2N} 
    \langle({\bf d}-\langle {\bf d}\rangle)^{\rm T}{\bf C}_{\boldsymbol\beta}^{-1}( {\bf d}-\langle {\bf d}\rangle)\rangle
    \\&+\frac 1 {2N} 
    (\langle {\bf d}\rangle_{\boldsymbol\beta}-\langle {\bf d}\rangle)^{\rm T}{\bf C}_{\boldsymbol\beta}^{-1}
    (\langle {\bf d}\rangle_{\boldsymbol\beta}-\langle {\bf d}\rangle)
    -\frac 1 2,
    \end{aligned}
    \label{eq:extensivecrossentropyforasymptoticnormality}
\end{equation}
which is obtained similarly to eq.~(\ref{eq:extensivecrossentropyforasymptoticnormality_conditional}). 
Equation~(\ref{eq:extensivecrossentropyforasymptoticnormality}) upholds the prerequisite for the model-selection efficiency of the MMLE in \S\ref{sec:gen_efficiency}.

The leading-order equivalence between the empirical and generalization  cross-entropy losses as above follows the asymptotic normality of the prior and true predictive distributions. 
Therefore, 
Lemma~\ref{le:genconsistencyandefficiencywithequivalenceincrossentropies} (with the argument in \S\ref{sec:gen_efficiency}) applicable to non-exponential families concludes 
the MMLE is there consistent and efficient when
the global minimum is well-defined in the leading-order functions of the cross entropy, respectively $h(\boldsymbol\beta)$ and $h^\prime(\boldsymbol\beta)$. 
\begin{theorem}[Consistency and model-selection efficiency of the MMLE for asymptotically normal predictives]\label{th:genconsistencyundePAN}
When $P({\bf d};\boldsymbol\beta)$ and $Q({\bf d})$ satisfy the asymptotic normality, 
the MMLE $\hat {\boldsymbol\beta}$ is consistent and model-selection efficient insofar as 
their normalized generalization cross-entropy loss $h^\prime(\boldsymbol\beta)=\lim_{N\to\infty}N^{-1}\langle -\ln P({\bf d};\boldsymbol\beta)\rangle$ has the single global minimum. 
\end{theorem}

The coverage of Theorem~\ref{th:genconsistencyundePAN} includes $P({\bf d};\boldsymbol\beta)$ and $Q({\bf d})$ belonging to
the two-(or one-)hyperparameter general linear models (eqs.~\ref{eq:twohyperparameterform}), where the cross entropy has a sufficiently convex unique stationary point (ineq.~\ref{eq:strongmodality}) as seen in \S\ref{eq:proof4gaussmodel}. 
Another example is the set of $P({\bf d};\boldsymbol\beta)$ and $Q({\bf d})$ that provides $\langle -\ln P({\bf d};\boldsymbol\beta)\rangle\propto N$, for which $h^\prime(\boldsymbol\beta)$ [comprehending $h(\boldsymbol\beta)$] becomes $\langle -\ln P({\bf d};\boldsymbol\beta)\rangle/N$ minimized at the minimum of $\langle -\ln P({\bf d};\boldsymbol\beta)\rangle$ (i.e. $\boldsymbol\beta_0$ and $\boldsymbol\beta_0^\prime$). 
$\langle -\ln P({\bf d};\boldsymbol\beta)\rangle\propto N$ holds in parametric models ($M=const.$) for independent data and also in nonparametric models ($M/N=const.$) with recursive structures often seen in the partial differential equation modeling. The following includes a one-dimensional case for the latter. 

\begin{example}
The cross entropy $\langle-\ln P({\bf d};\boldsymbol\beta) \rangle $ of $P({\bf d};\boldsymbol\beta)$ relative to $Q({\bf d})$ is proportional to $N$ and the global minimum (minima) of $h^\prime(\boldsymbol\beta)$ is (are) $\boldsymbol\beta_0^\prime$ 
for independent data such that $P({\bf d}|{\bf a};\boldsymbol\beta)=\prod_i P(d_i|{\bf a};\boldsymbol\beta)$ and $Q({\bf d})=\prod_i Q(d_i)$ in a model with fixed $M$. 
The same holds in a model of constant $M/N\in\mathbb N$ under the model translation symmetry: $P({\bf d},{\bf a})=P({\bf d},{\bf a})|_{{\bf d}=\boldsymbol{\mathcal C}{\bf d},{\bf a}=\boldsymbol{\mathcal C}^{M/N}{\bf a}}$, where $\boldsymbol{\mathcal C}$ denotes the circulant matrix.
\end{example}

The asymptotic normality of the posterior in the inverted and true models is the basis of the derivation of AIC and BIC~\citep{akaike1974new,schwarz1978estimating}. 
Besides AIC and BIC assume the data independent of each other, 
which corresponds to the above case of $\langle -\ln P({\bf d};\boldsymbol\beta)\rangle\propto N$ with finite $M$, yielding the consistency and model-selection efficiency of the MMLE in the hyperparameter inference under the asymptotic normality of the predictive distributions; we note the proofs in \S\ref{sec:gen_efficiency} that have focused on the $M\propto N$ cases can be repeated with constant $M$. 
The asymptotic normality, giving different criteria in the dimension decision of parametric models, 
provides a unified representation for the minimum empirical and generalization cross-entropy loss criteria in the hyperparameter inference, which is the MMLE. 

\section{Discussion and conclusions}

We have shown the consistency of the MMLE for the natural parameters in a range of the exponential families, which cover various applications such as the generalized linear model~\citep{dobson2018introduction}, or more specifically the Gibbs distributions, as well as more generally for the hyperparameter determination in the non-exponential families with asymptotically normal data predictive distributions. 
The consistency of the MMLE is applicable to $M/N>1$, where the maximum likelihood estimate is ill-posed. 
We detailed the proof in a general linear model and extended to various cases with highlighting the essence. 
An intrinsic postulate of the proof is found in the extensive properties of the cost functions that hold in various Gibbsian models. 
When the unknown hyperparameters are not included in the model-parameter prior, 
the extensive property of the cost functions in the model-parameter likelihood is a unique postulate of the consistency below the number-ratio ($M/N$) cutoff of the model parameters to the data. 
Another key relation is the leading-order equivalence between the empirical and generalization cross-entropy losses, 
and under the asymptotic normality of the predictive distribution applicable to the detailed two-hyperparameter general linear models, more widely the model-selection efficiency holds with applicability to the true data distribution outside the model space. 
Our proof of consistency and model-selection efficiency also verifies the multi-stage inference of the model parameters involving the hyperparameter point estimation of the MMLE, namely the empirical Bayes method, which has been experimentally successful yet often lacking theoretical foundations. 
We find the empirical Bayes method allows the (conditional) posterior of the model parameters combines the stability and asymptotic unbiasedness. Examples include the regularized least square estimate using the hyperparameters of the MMLE, or here equivalently, ABIC. 
Because the marginal likelihood of the hyperparameters is the empirical cross entropy for a single data set, 
these proofs for the consistency and model-selection efficiency of the MMLE ensure the same also for the method of minimum (empirical) cross entropy common in cross-validations~\citep{friedman2001elements}.

One robust result is the asymptotic almost sure leading-order stationarity (eqs.~\ref{eq:generalalmostsurealmoststationarity_costfunctionform} and \ref{eq:generalalmostsurealmoststationarity}) of the marginal likelihood of the hyperparameters at the true hyperparameter set. 
It immediately follows the extensive properties of the cost functions (eq.~\ref{eq:assumptionofalmostsurealmoststationarity}) as Lemma~\ref{le:almostsurealmoststationarity}, thus expected to be universal. 
Lemma~\ref{le:almostsurealmoststationarity} implies the true hyperparameter set may be close to at least either of the stationary points in the marginal likelihood of the hyperparameters when the extensive properties are satisfied by the mean and variances of the conjugated cost functions; at least, the MMLE is consistent for the globally sufficiently concave marginal likelihood of the hyperparameters  (Theorem~\ref{th:genconsistencywithsufficientconvexity}). 
Limitation of Lemma~\ref{le:almostsurealmoststationarity} may arise from the first-order phase transition that violates the extensive property of the cost-function variance. 

The extensive properties of the cost functions (eq.~\ref{eq:assumptionofalmostsurealmoststationarity}) premised in 
Lemma~\ref{le:almostsurealmoststationarity} 
also yield
\begin{equation}
\begin{aligned}
P({\bf U}/N,{\bf V}/M;\boldsymbol\beta)
&\to \delta({\bf U}/N-\langle {\bf U}/N\rangle_{\boldsymbol\beta})\delta({\bf V}/M-\langle {\bf V}/M\rangle_{\boldsymbol\beta}),
\\
P({\bf U}/N,{\bf V}/M|{\bf d};\boldsymbol\beta)
&\to \delta({\bf U}/N-\langle {\bf U}/N\rangle_{\boldsymbol\beta|{\bf d}})\delta({\bf V}/M-\langle {\bf V}/M\rangle_{\boldsymbol\beta|{\bf d}}).
\end{aligned}
\end{equation}
That is, the posterior $P({\bf a}|{\bf d};\boldsymbol\beta)$ giving reasonable extensive cost function values in view of the original statistical model $P({\bf a},{\bf d};\boldsymbol\beta)$ is limited to the stationary point $\boldsymbol\beta_*\ni\boldsymbol\beta_0,\boldsymbol\beta_0^\prime$ of the marginal likelihood of the hyperparameters:
\begin{equation}
    \mathcal D[P(\boldsymbol{\mathcal E};\boldsymbol\beta)||P(\boldsymbol{\mathcal E}|{\bf d};\boldsymbol\beta)]\to 
        \infty  \hspace{4pt}(N\to\infty;\boldsymbol\beta\not\in\boldsymbol\beta_*).
\end{equation}
Here we can reasonably request the observation invariance of the cost function mean (eq.~\ref{eq:Ibarel}), that is the extremal representation of the MMLE. 
Equation~(\ref{eq:Ibarel}) is actually true on average as the mean invariance under the probability decomposition (eq.~\ref{eq:avIbarel2}) 
and derives from the free entropy difference nature of the marginal likelihood of the hyperparameters (eq.~\ref{eq:logPPMasfreenetropydiff}). 
Several studies have pointed out the correspondence between the marginal likelihood of the hyperparameters (the model evidence) and free entropy~\citep[e.g.,][]{iba1989bayesian,watanabe2018mathematical}. Yet, the model evidence may require further investigation as the free entropy difference. 

The assumption of a unique stationary point is unnecessary in the present proof of consistency of the MMLE. Even the proof of the model-selection efficiency using the asymptotic normality of the predictive distribution with respect to data in \S\ref{sec:gen_efficiency} covers the marginal likelihood of the hyperparameters with multiple hyperparameter stationary points. 
Meanwhile, the analysis in \S\ref{eq:proof4gaussmodel} proves as Theorem~\ref{th:uniquestationarityingenlinearmodels} the 
asymptotic almost sure uniqueness of the stationary point in the marginal likelihood of the hyperparameters for the two-hyperparameter general linear models (eqs.~\ref{eq:twohyperparameterform}). 
We also saw the marginal likelihood of the hyperparameters is concave with respect to the hyperparameters of the model-parameter likelihood conjugated with extensive cost functions below a critical point of the number ratio $N/M$, 
even with the nonlinear multimodal model-parameter likelihood and prior without the asymptotic normality (Lemma~\ref{le:convexityofthemarginalposteriorwrtmodelparameterlikelihoodhyperparameters}; \S\ref{sec:gen2}). 
These suggest the marginalization may invest the likelihood with a well-behaved probability landscape. 
A fully Bayesian analysis~\citep{sato2022appropriate} of the two-hyperparameter general linear models (eqs.~\ref{eq:twohyperparameterform}) 
actually finds the well-behavedness of the marginal posterior of the hyperparameters contrasted with the pathology in the joint posterior of the model parameters and hyperparameters, which causes the well-posedness of the ABIC estimate despite the ill-posedness of the maximum a posteriori estimate. 
However, those preferable natures of the marginal likelihood would not apply to any parameter subsets, since this study and Ref.~\citep{sato2022appropriate} owe most of the results to the special properties of the natural parameters in the exponential family. 
Another exception may be within the empirical Bayes approach using an infinite number of hyperparameters [s.t. ${\rm dim}(\boldsymbol\beta)/N=const.\cup \infty$], as this study and Ref.~\citep{sato2022appropriate} also utilize the finiteness of the hyperparameter dimension.

As is known in the decision problem of the model-parameter dimension, the consistency in selecting the true model-parameters is often incompatible with the model-selection efficiency in approximating the true data generative distribution~\citep{watanabe2018mathematical,yang2005can}. 
AIC aims the minimum Kullback-Leibler divergence between the prior and true predictive distributions (i.e. the minimum generalization cross-entropy loss, $\langle -\ln P({\bf d};\boldsymbol\beta) \rangle$)~\citep{akaike1974new} and is model-selection efficient and minimax-rate optimal~\citep{barron1999risk} in the squared error loss minimization for the true data distribution outside the model space~\citep{shibata1981optimal,shibata1983asymptotic}. 
BIC follows the MMLE (minimum $-\ln P({\bf d};\boldsymbol\beta)$)~\citep{schwarz1978estimating} and is consistent within the scope of its approximation~\citep{nishii1984asymptotic}. 
Since AIC and BIC select different numbers of model parameters, the model-selection efficiency and consistency are not satisfied at once in the dimension decision of the parametric models. 
The closeness to the true distribution in the data prediction is there another direction of optimization from the knowledge acquisition on the true parameters~\citep{watanabe2018mathematical}. 
In contrast, ABIC derived from the minimum Kullback-Leibler divergence similarly to AIC is equivalent to the MMLE yielding BIC in the inference of a fixed number of hyperparameters. 
Therefore in these models, the data predictability and knowledge acquisition on the true parameters disable their difference in optimization functions, 
and as such, the predictability and knowledge acquisition are naturally likely to be compatible in well-posed hierarchical inferences. 
This study in fact shows even the cross entropy for one data set duplicates the generalization cross entropy in its leading order with probability 1, in a range of Gibbsian models or under the asymptotic normality of the data predictive distribution. 
It verifies the MMLE as a consistent besides model-selection efficient criterion, that is as an optimum both for the data predictability and knowledge acquisition, and thus would found the experimental successfulness of the empirical Bayes method.





\begin{funding}
This study was partly supported by JSPS KAKENHI Grant Number 21J01694. 
\end{funding}

\begin{supplement}
\stitle{Derivation of ineq.~(\ref{eq:strongmodality})}
\sdescription{The attached supplement presents the derivation of ineq.~(\ref{eq:strongmodality}).}
\end{supplement}


\bibliographystyle{imsart-number} 


\section*{Derivation of ineq.~(\ref{eq:strongmodality})}
\renewcommand{\theequation}{S.\arabic{equation}}
\setcounter{equation}{0}
Inequality~(\ref{eq:strongmodality}) is derived in the $\boldsymbol\beta^\prime$ space. 
\footnote{
The following derivation is intended to appear in a supplemental file: Ref.~\citep{sato2022supplement}. 
}
For $|\beta_{\rm d}-\tilde\beta(\lambda)|>\tilde\epsilon_{\beta_{\rm d}}$, 
\begin{equation}
\begin{aligned}
    \langle\ln L(\boldsymbol\beta)\rangle
    \leq& \max_{\pm=\pm1}\langle\ln L(\tilde\beta_{\rm d}(\lambda)\pm\tilde\epsilon_{\beta_{\rm d}},\lambda)\rangle
    \\=&
    \langle\ln L(\boldsymbol\beta_0)\rangle
    +\max_{\pm=\pm1}
    \int^{\tilde\epsilon_{\beta_{\rm d}}}_0 d \Delta\beta_{\rm d} \int^{\Delta\beta_{\rm d}}_0 d \Delta\beta_{\rm d}^\prime \left(\frac{\partial^2\langle\ln L(\tilde\beta_{\rm d}(\lambda)+\Delta\beta_{\rm d}^\prime,\lambda)\rangle}{\partial(\Delta\beta_{\rm d}^\prime)^2}\right)_\lambda
    \\\leq&
    \langle\ln L(\boldsymbol\beta_0)\rangle
    +\max_{|\beta_{\rm d}-\tilde\beta(\lambda)|\leq\tilde\epsilon_{\beta_{\rm d}}} \left(\frac{\partial^2\langle\ln L(\boldsymbol\beta)\rangle}{\partial\beta_{\rm d}^2}\right)_\lambda\frac{\tilde\epsilon_{\beta_{\rm d}}^2}{2}
    \\\leq&
    \langle\ln L(\boldsymbol\beta_0)\rangle- \frac 1 2 \frac{N+P-M }{(\tilde\beta_{\rm d}+\tilde\epsilon_{\beta_{\rm d}})^2}\tilde\epsilon_{\beta_{\rm d}}^2;
\end{aligned}
\label{eq:lnPboundbetad}
\end{equation}
the first transform utilizes the concavity of $\ln P({\bf d};\boldsymbol\beta)$ with fixed $\lambda$ values (eq.~\ref{eq:secondderivinbetaprimespace}), and the second to fourth ones evaluate the Taylor series using eq.~(\ref{eq:secondderivinbetaprimespace}). 
For $|\lambda-\lambda_0|>\epsilon_\lambda$, 
\begin{equation}
\begin{aligned}
    \langle\ln L(\boldsymbol\beta)\rangle
    \leq& \langle\ln L(\tilde\beta_{\rm d}(\lambda),\lambda)\rangle
    \leq
    \max_{\pm=\pm1}\langle\ln L(\tilde\beta_{\rm d}(\lambda_0\pm\epsilon_\lambda),\lambda_0\pm\epsilon_\lambda)\rangle
    \\\leq&
    \langle\ln L(\boldsymbol\beta_0)\rangle
    +\max_{\Delta\lambda\leq\epsilon_\lambda} \frac{d^2\langle\ln L(\tilde\beta_{\rm d}(\lambda_0+\Delta\lambda^\prime),\lambda_0+\Delta\lambda^\prime)\rangle}{d(\Delta\lambda^\prime)^2}\frac{\epsilon_\lambda^2}{2}
    \\\leq&
    \langle\ln L(\boldsymbol\beta_0)\rangle
    -\frac 1 2 \frac{c_{\rm det}(\lambda_0)c_{\rm TS}^\prime(N+P-M)}{2\lambda_0^2}\epsilon_\lambda^2,
\end{aligned}
\label{eq:lnPboundlambda}
\end{equation}
where
\begin{equation}
\begin{aligned}
    c_{\rm TS}^\prime&:=\max_{\Delta\lambda\leq\epsilon_\lambda} \frac{d^2\langle\ln L(\tilde\beta_{\rm d}(\lambda_0+\Delta\lambda^\prime),\lambda_0+\Delta\lambda^\prime)\rangle}{d(\Delta\lambda^\prime)^2}\left/\frac{d^2\langle\ln L(\tilde\beta_{\rm d}(\lambda),\lambda)\rangle}{d\lambda^2}\right|_{\lambda=\lambda_0}
    \\&=1+\mathcal O(\epsilon_{\lambda});
\end{aligned}
\end{equation}
the first transform utilizes the concavity of $\langle\ln P({\bf d};\boldsymbol\beta)\rangle$ with fixed $\lambda$ values (eq.~\ref{eq:secondderivinbetaprimespace}), 
the second transform utilizes the monotonicity of $\langle\ln P({\bf d};\tilde\beta_{\rm d},\lambda)\rangle$ within $\lambda<\lambda_0$ and $\lambda>\lambda_0$, derived from the uniqueness of the stationary point in $\langle\ln P({\bf d};\tilde\beta_{\rm d},\lambda)\rangle$ (eq.~\ref{eq:uniquestationarity_lambda}), and the left part evaluates the Taylor series with eq.~(\ref{eq:differentiatedprofilelikelihoodlikefunc}). 
The following differentiated form of  eq.~(\ref{eq:differentiatedprofilelikelihoodlikefunc}) may be useful: 
\begin{equation}
    \begin{aligned}
    &
    \partial^2_\lambda\langle\ln P({\bf d};\tilde \beta_{\rm d},\lambda)\rangle_{\boldsymbol\beta_0}
    \\=&
    \{2-{\rm Tr}[{\bf D}({\bf I}-{\bf D}){\bf D}_0^\dag]/{\rm Tr}[{\bf D}{\bf D}_0^\dag]\}\partial_\lambda\langle\ln P({\bf d};\tilde \beta_{\rm d},\lambda)\rangle_{\boldsymbol\beta_0}/\lambda
    \\&-
    \frac{N^\prime}{2\lambda^2}\frac{2{\rm Tr}[{\bf D}^3{\bf D}_0^\dag]-{\rm Tr}[{\bf D}^2{\bf D}_0^\dag]{\rm Tr}[{\bf D}]/N^\prime-{\rm Tr}[{\bf D}{\bf D}_0^\dag]{\rm Tr}[{\bf D}^2]/N^\prime}{{\rm Tr}[{\bf D}{\bf D}_0^\dag]}
    \\
    \to&-
    \frac{N^\prime c_{\rm det}}{2\lambda^2}\hspace{4pt} (\lambda\to\lambda_0).
    \end{aligned}
\end{equation}


\pagestyle{empty}

We use the above inequalities with mapping $|\boldsymbol\beta-\boldsymbol\beta_0|>\epsilon$ to its superset rectangle in the $\boldsymbol\beta^\prime$ space $S_{\beta^\prime}(\epsilon_{\beta_{\rm d}},\epsilon_\lambda):=|\beta_{\rm d}-\beta_{\rm d0}|>\epsilon_{\beta_{\rm d}} \cup|\lambda-\lambda_0|>\epsilon_{\lambda}$. 
The sides $\epsilon_{\beta_{\rm d}}$ and $\epsilon_\lambda$ of the rectangle $S_{\beta^\prime}(\epsilon_{\beta_{\rm d}},\epsilon_\lambda)$ are constrained as follows: 
\begin{equation}
    \begin{aligned}
    &|\boldsymbol\beta-\boldsymbol\beta_0|>\epsilon\in 
    |\beta_{\rm d}-\beta_{\rm d0}|>\epsilon_{\beta_{\rm d}} \cup|\lambda-\lambda_0|>\epsilon_{\lambda}
    \\
    \Leftarrow&
    |\boldsymbol\beta-\boldsymbol\beta_0|\leq \epsilon\ni 
    (\beta_{\rm d},\lambda)=(\beta_{\rm d0}+\epsilon_{\beta_{\rm d}},\lambda_0+\epsilon_{\lambda})
    \\
    \Leftrightarrow&
    \epsilon^2\geq \epsilon_{\beta_{\rm d}}^2+(\lambda_0\epsilon_{\beta_{\rm d}}+\beta_{\rm d0}\epsilon_\lambda+\epsilon_{\beta_{\rm d}}\epsilon_\lambda)^2.
    \end{aligned}
    \label{eq:betacirclebetaprimerectanglemapping}
\end{equation}
The first transform requires the apex of the quadrangle $|\beta_{\rm d}-\beta_{\rm d0}|=\epsilon_{\beta_{\rm d}} \cap|\lambda-\lambda_0|=\epsilon_{\lambda}$ farthest from $\boldsymbol\beta=\boldsymbol\beta_0$ to be within the $\epsilon$-neighborhood of $\boldsymbol\beta_0$. 
Then, we evaluate the $\ln P({\bf d};\boldsymbol\beta)$ value in $S_{\beta^\prime}$ separately for $|\lambda-\lambda_0|>\epsilon_\lambda$ where ineq.~(\ref{eq:lnPboundlambda}) holds and 
$|\beta_{\rm d}-\beta_{\rm d0}|>\epsilon_{\beta_{\rm d}}\cap |\lambda-\lambda_0|\leq\epsilon_\lambda$. 
The latter case 
$|\beta_{\rm d}-\beta_{\rm d0}|>\epsilon_{\beta_{\rm d}}\cap |\lambda-\lambda_0|\leq\epsilon_\lambda$ utilizes
the $\beta_{\rm d}$ bound $|\beta_{\rm d}-\tilde \beta_{\rm d}|>\tilde \epsilon_{\beta_{\rm d}}$, considering
\begin{equation}
\begin{aligned}
    |\beta_{\rm d}-\tilde \beta_{\rm d}|
    \geq& |\beta_{\rm d}-\beta_{\rm d0}|-|\tilde \beta_{\rm d}-\beta_{\rm d0}|
    \\
    >&\epsilon_{\beta_{\rm d}}-\max_{|\lambda-\lambda_0|\leq\epsilon_\lambda} \left|\frac{\partial\tilde\beta_{\rm d}}{\partial\lambda}\right|\epsilon_\lambda,
\end{aligned}
\end{equation}
which allows us to use ineq.~(\ref{eq:lnPboundbetad}) with the following form of $\tilde \epsilon_{\beta_{\rm d}}$:
\begin{equation}
    \tilde \epsilon_{\beta_{\rm d}}=\epsilon_{\beta_{\rm d}}-\max_{|\lambda-\lambda_0|\leq\epsilon_\lambda} \left|\frac{\partial\tilde\beta_{\rm d}}{\partial\lambda}\right|\epsilon_\lambda>0.
\label{eq:defoftildeepsilonbetad}
\end{equation}
The use of the above $\tilde \epsilon_{\beta_{\rm d}}$ (eq.~\ref{eq:defoftildeepsilonbetad}) constrains the $\epsilon_{\beta_{\rm d}}$ and $\epsilon_\lambda$ values,
and here we adopt 
\begin{equation}
    \epsilon_\lambda =\frac 1 2 \max_{|\lambda-\lambda_0|\leq\epsilon_\lambda}\left|\frac{d \tilde \beta_{\rm d}}{d \lambda}\right|^{-1}\epsilon_{\beta_{\rm d}}.
    \label{eq:selectedepsilonlambda}
\end{equation}
For this $\epsilon_\lambda$ (eq.~\ref{eq:selectedepsilonlambda}) with $\epsilon_{\rm \beta_{\rm d}}<\beta_{\rm d0}$, logical formula~(\ref{eq:betacirclebetaprimerectanglemapping}) bounds $\epsilon_{\beta_{\rm d}}$ as 
\begin{equation}
\begin{aligned}
    &
    \epsilon_{\beta_{\rm d}}/\epsilon\leq  \left(1+\left(\lambda_0+\frac{\beta_{\rm d0}+\epsilon_{\beta_{\rm d}}}{2\max_{|\lambda-\lambda_0|\leq\epsilon_\lambda}\left|\frac{d \tilde \beta_{\rm d}}{d \lambda}\right|}\right)^2\right)^{-1/2}
    \\
    \Leftarrow&
    \epsilon_{\beta_{\rm d}}/\epsilon\leq  \left(1+\left(\lambda_0+\frac{\beta_{\rm d0}}{\max_{|\lambda-\lambda_0|\leq\epsilon_\lambda}\left|\frac{d \tilde \beta_{\rm d}}{d \lambda}\right|}\right)^2\right)^{-1/2}.
\end{aligned}
\end{equation}
Here we select
\begin{equation}
    \epsilon_{\beta_{\rm d}}=  \left(1+\lambda_0+\beta_{\rm d0}\max_{|\lambda-\lambda_0|\leq\epsilon_\lambda}\left|\frac{d \tilde \beta_{\rm d}}{d \lambda}\right|^{-1}\right)^{-1}\frac {\epsilon}{\sqrt 2}.
\end{equation}
Further using eqs.~(\ref{eq:defoftildeepsilonbetad}) and (\ref{eq:selectedepsilonlambda}),
we have
\begin{equation}
\begin{aligned}
    \tilde\epsilon_{\beta_{\rm d}}&=  \left(1+\lambda_0+\beta_{\rm d0}\max_{|\lambda-\lambda_0|\leq\epsilon_\lambda}\left|\frac{d \tilde \beta_{\rm d}}{d \lambda}\right|^{-1}\right)^{-1}\frac {\epsilon}{2\sqrt 2}
    \\
    \epsilon_{\lambda}&=  \left(\beta_{\rm d0}+(1+\lambda_0)\max_{|\lambda-\lambda_0|\leq\epsilon_\lambda}\left|\frac{d \tilde \beta_{\rm d}}{d \lambda}\right|\right)^{-1}\frac {\epsilon}{2\sqrt 2}.
\end{aligned}
\label{eq:selectedepsilonbetadlambda_unrounded}
\end{equation}
For brevity, we last round $\tilde\epsilon_{\beta_{\rm d}}$ and $\epsilon_{\lambda}$: 
\begin{equation}
\begin{aligned}
    \tilde\epsilon_{\beta_{\rm d}}&=  \left(1+\lambda_0+\beta_{\rm d0}\left|\frac{d \tilde \beta_{\rm d}}{d \lambda}\right|_{\lambda=\lambda_0}^{-1}\right)^{-1}\frac {\epsilon}{3}
    \\
    \epsilon_{\lambda}&=  \left(\beta_{\rm d0}+(1+\lambda_0)\left|\frac{d \tilde \beta_{\rm d}}{d \lambda}\right|_{\lambda=\lambda_0}\right)^{-1}\frac {\epsilon}{3},
\end{aligned}
\label{eq:selectedepsilonbetadlambda}
\end{equation}
where we consider sufficiently small $\epsilon$ for the replacement of the denominator, such that
\begin{equation}
\begin{aligned}
    &\frac {2\sqrt 2}{3}\leq \left(1+\lambda_0+\beta_{\rm d0}\left|\frac{d \tilde \beta_{\rm d}}{d \lambda}\right|_{\lambda=\lambda_0}^{-1}\right)\left(1+\lambda_0+\beta_{\rm d0}\max_{|\lambda-\lambda_0|\leq\epsilon_\lambda}\left|\frac{d \tilde \beta_{\rm d}}{d \lambda}\right|^{-1}\right)^{-1}
    \left(=1+\mathcal O(\epsilon)\right)
    \\
    &\frac {2\sqrt 2}{3}\leq 
    \left(\beta_{\rm d0}+(1+\lambda_0)\left|\frac{d \tilde \beta_{\rm d}}{d \lambda}\right|_{\lambda=\lambda_0}\right)\left(\beta_{\rm d0}+(1+\lambda_0)\max_{|\lambda-\lambda_0|\leq\epsilon_\lambda}\left|\frac{d \tilde \beta_{\rm d}}{d \lambda}\right|\right)^{-1}
    \left(=1+\mathcal O(\epsilon)\right),
\end{aligned}
\end{equation}
which ensures
$S_{\beta^\prime}(\epsilon_{\beta_{\rm d}},\epsilon_\lambda)$ of eqs.~(\ref{eq:selectedepsilonbetadlambda}) includes
$S_{\beta^\prime}(\epsilon_{\beta_{\rm d}},\epsilon_\lambda)$ of eqs.~(\ref{eq:selectedepsilonbetadlambda_unrounded}) and thus also $|\boldsymbol\beta-\boldsymbol\beta|>\epsilon$. 
Substituting eqs.~(\ref{eq:selectedepsilonbetadlambda}) 
into 
ineqs.~(\ref{eq:lnPboundbetad}) and (\ref{eq:lnPboundlambda}) 
with 
the derivative of $\tilde \beta_{\rm d}$ (eq.~\ref{eq:implicativebetadform}),
\begin{equation}
    \frac{d\tilde\beta_{\rm d}}{d\lambda}
    = 
    -\frac{\beta_{\rm d0}}{\lambda}\frac{N^\prime{\rm Tr}[({\bf I}-{\bf D}){\bf DD}_0^\dag]}{[{\rm Tr}({\bf DD}_0^\dag)]^2} 
    \to
    -\frac{\beta_{\rm d0}}{\lambda_0}\frac{{\rm Tr}[({\bf I}-{\bf D}_0){\bf D}_0{\bf D}_0^\dag]}{N^\prime}
    \hspace{4pt}(\lambda\to\lambda_0),
\end{equation}
we find ineq.~(\ref{eq:strongmodality}); 
for sufficiently small $\epsilon$, for when
$|\beta_{\rm d}-\beta_{\rm d0}|>\epsilon_{\beta_{\rm d}}\cap |\lambda-\lambda_0|\leq\epsilon_\lambda$,
\begin{equation}
    \left(\frac{\tilde\epsilon_{\beta_{\rm d}}/\epsilon}{\tilde \beta_{\rm d}+\tilde \epsilon_{\beta_{\rm d}}}\right)^2
    \geq \left(\frac{\tilde\epsilon_{\beta_{\rm d}}/\epsilon}{2\beta_{\rm d0}}\right)^2
    =\left[6\beta_{\rm d0}\left(1+\lambda_0+\frac{N^\prime \lambda_0}{{\rm Tr}[({\bf I}-{\bf D}_0){\bf D}_0{\bf D}_0^\dag]}\right)\right]^{-2},
\end{equation}
and for when $|\lambda-\lambda_0|>\epsilon_\lambda$,
\begin{equation}
    \frac{c_{\rm det}c_{\rm TS}^\prime}{2\lambda_0^2}(\epsilon_\lambda/\epsilon)^2
    \geq\frac{c_{\rm det}}{3\lambda_0^2}(\epsilon_\lambda/\epsilon)^2
    =\frac{c_{\rm det}}{27}[\beta_{a0}+(\beta_{\rm d0}+\beta_{\rm a0}){\rm Tr}[({\bf I}-{\bf D}_0){\bf D}_0{\bf D}_0^\dag]/N^\prime]^{-2}.
\end{equation}

\end{document}